\documentclass{amsart}
\usepackage{amsfonts}
\usepackage{amsmath}
\usepackage{tikz-cd}
\usepackage{graphicx}
 \usepackage{algorithmic} 
\usepackage{algorithm}
\usepackage{rotating}
\usepackage{hyperref}
 \usepackage{amssymb}

\newtheorem{theorem}{Theorem}[section]
\newtheorem{lemma}[theorem]{Lemma}
\newtheorem{remark}[theorem]{Remark}
\newtheorem{example}[theorem]{Example}
\newtheorem{definition}[theorem]{Definition}
\newtheorem{lemma-definition}[theorem]{Lemma-Definition}

\newtheorem{corollary}[theorem]{Corollary}
\newtheorem{proposition}[theorem]{Proposition}
\newtheorem{lem-def}[theorem]{Lemma-Definition}

\newcommand{\R}{\mathbb R}

\newcommand{\Z}{\mathbb Z}
\newcommand{\Q}{\mathbb Q}

\newcommand{\F}{\mathbb F}

\def\op{\operatorname}

\def\as#1{\renewcommand\arraystretch{#1}}
\def\aut{\op{Aut}}

\def\bb{{\mathcal B}}

\def\sp\operatorname{Spec}

\def\dsc{\mathrm{Disc}}
\def\diso{\lower.4ex\hbox{$\downarrow$}\raise.4ex\hbox{\mbox{\scriptsize
$\wr$}}}

\def\gen#1{\big\langle\, {#1} \,\big\rangle}
\def\gl#1#2{\op{GL}_{#1}(#2)}

\def\iso{\ \lower.3ex\hbox{\as{.08}$\begin{array}{c}\lra\\\mbox{\tiny $\sim\,$}\end{array}$}\ }

\def\lg{l\raise.6ex\hbox to.2em{\hss.\hss}l}
\def\ll{\mathcal{L}}

\def\lra{\longrightarrow}
\def\m{{\mathfrak m}}
\def\md#1{\; \mbox{\rm(mod }{#1})}

\def\oo{\mathcal{O}}
\def\orb{\hbox to  .3em{$\backslash$}\backslash}

\def\p{\mathfrak{p}}

\def\q{\mathfrak{Q}}

\def\sm{S_{\operatorname{max}}}

\def\sii{\,\Longleftrightarrow\,}

\def\t{\theta}

\def\vv{\Vert~\Vert}

\def\dd{\mathcal{D}}
\def\v#1{\Vert#1\Vert}
\def\V#1{\Big\Vert#1\Big\Vert}
\def\vol{\operatorname{vol}}

\def\zc{\operatorname{zc}}
\def\pp{\mathbb{P}}
\def\sm{\operatorname{sm}}
\def\smb{\overline{\operatorname{sm}}}
\def\red{\operatorname{red}}
\def\bb{{\mathcal B}}
\def\cc{{\mathcal C}}
\def\dd{{\mathcal D}}
\def\od{\operatorname{OD}}
\def\ii{\mathcal{I}}

\def\cf{{C_f}}

\def\gl{\mathrm{GL}}

\newcounter{cs}
\stepcounter{cs}
\newcommand{\fcasos}{\end{itemize}\setcounter{cs}{1}}

\newfont{\tit}{cmr12 scaled \magstep3}

\def\ww#1_#2{w_#2(#1)}
\def\WW#1_#2{w_#2\Big(#1\Big)}

\def\kk{\mathcal{K}}

\def\ki{K_\infty}

\def\oi{{\mathcal O}_{F,\infty}}
\def\red{\operatorname{red}}

\def\tp{\mathrm{tr}}
\def\o1{{\mathcal O}_F}

\def\dia{\mathrm{diag}}
\def\sig{\operatorname{Sig}}
\def\lcm{\mathrm{lcm}}
\def\NM{\mathrm{Norm}}
\def\sm{\operatorname{sm}}
\def\LS{\mathrm{LS}}

\def\ma{\mathrm{Max}}
\def\Tr{\mathrm{Tr}_{F/K}}

\title{Lattices over Polynomial Rings and applications to function fields}

\thanks{This research was supported by MTM2013-40680-P from the
Spanish MEC and by the Netherlands Organization for 
Scientific Research (NWO) under grant 613.001.011.}

\author{Jens-Dietrich Bauch}
\address{Math Department - Simon Fraser University
SCK 10502 - 8888 University Drive
Burnaby, BC    V6C 1A5, Canada}
\email{j.bauch@sfu.ca}

\makeatletter

\begin{document}

\begin{abstract}
This paper deals with lattices $(L,\vv)$ over polynomial rings, where $L$ is a finitely generated module over $k[t]$, the polynomial ring over the field $k$ in the indeterminate $t$, and $\vv$ is a discrete real-valued length function on $L\otimes_{k[t]}k(t)$. A reduced basis of $(L,\vv)$ is a basis of $L$ whose vectors attain the successive minima of $(L,\vv)$. We develop an algorithm which transforms any basis of $L$ into a reduced basis of $(L,\vv)$. By identifying a divisor $D$ of an algebraic function field with a lattice $(L,\vv)$ over a polynomial ring, this reduction algorithm can be addressed to the computation of the Riemann-Roch space of $D$ and the successive minima of $(L,\vv)$, without the use of any series expansion.

\end{abstract}
\maketitle

\section*{Introduction}
The theory of lattices over the integers is an important tool in algebraic number theory. Lattices over the polynomial ring $k[t]$ in an indeterminate $t$, over a field $k$, admit a similar development although the theory becomes simpler. For instance, a shortest vector in a lattice can be found in polynomial time, whereas this problem shall be deemed to be difficult in a lattice over $\Z$. 

The theory of lattices over $k[t]$ is in substance due to Mahler \cite{MA}. Lattices over polynomial rings (or Puiseux series rings) are used to factorize multivariate polynomials \cite{Len} and to compute Riemann-Roch spaces in algebraic function fields \cite{F.H.}, \cite{Sch}, \cite{Schoe}. The idea of constructing bases of Riemann-Roch spaces of a divisor $D$ by computing vectors of short length in a lattice $(L,\vv)$ corresponding to $D$ is due to W. M. Schmidt \cite{Sch}. His method is based on the computation of Puiseux series in the context of function fields in one variable over number fields. This idea was adopted by M. Sch\"ornig \cite{Sch} to global function fields, which are tamely ramified at the places at infinity. Both methods use series expansions, which result in several technical problems; e.g. constant field extension are necessary and it has to take care that the series are computed to enough precision. F. Hess \cite{F.H.} could solve these problems by identifying a divisor $D$ with a 
simplified lattice $(L',\v{~}')$, which yields an algorithm that avoids series expansions and applies to function fields over arbitrary (``computable") constant fields. However, the simplified lattice $(L',\v{~}')$ (and therefore Hess' algorithm) does not carry out the successive minima of the original lattice attached to $D$, only approximations.

The theory of lattices over $k[t]$ plays an important role in coding theory and cryptanalysis in the context of convolutional codes \cite{Code} and in the computation of approximated common divisors over polynomial rings \cite{HCNH}.

In all these settings it is necessary to determine a reduced basis (cf. Section \ref{subredbasis}) of a lattice. This led to several reduction algorithms \cite{Len,Gath,Sch,Schoe,MS}, which transform any basis of a lattice into a reduced one. While these methods cover particular cases, we present a reduction algorithm which determines a reduced basis in a general setting (cf. Section \ref{reductionalgorithm}) and, applied to the computation of Riemann-Roch spaces, it fixes the flaw of Hess' algorithm; that is, we are able to compute the Riemann-Roch space of a divisor $D$ and the successive minima of the corresponding lattice without any series expansions.

The article is divided in the following sections. In Section 1 we introduce general lattices, their successive minima, and normed spaces. We define the concept of reduced bases and prove their existence in any lattice (cf. Lemma \ref{exist}). Moreover, we define length preserving maps between normed spaces (isometries) and compute the general structure of the isometry group of a normed space. In Section 2 we introduce a reduction algorithm, which transforms any basis of a non integral-valued lattice into a reduced one. It generalizes the classical approach of A. Lenstra for integral-valued lattices \cite{Len}, to the non integral-valued case. In Section 3 we consider the computation of Riemann-Roch spaces of divisors of function fields (and their successive minima) as an application of the new reduction algorithm. In Section 4 we give a precise estimation of the complexity of this method.

\section{Lattices and normed spaces}
Let $k$ be a field and denote by $A=k[t]$, $K=k(t)$, 
the polynomial ring and the rational function field in the indeterminate $t$ over $k$, respectively. 

For any rational function $x=a/b\in K$, where $a,b\in A$ and $b\ne0$, we define
$$
v_\infty(x)=\left\{\begin{array}{ll}
\deg b-\deg a,&\mbox{ if }x\ne0,\\
\infty,&\mbox{ if }x=0.
\end{array}
\right.
$$ 
This is a discrete valuation on $K$, with valuation ring
 $A_\infty=k[t^{-1}]_{(t^{-1})}\subset K$ and maximal ideal $P_\infty=\m_\infty=t^{-1}A_\infty$. We denote by $U_\infty=\{a\in K\mid v_\infty(a)=0\}$ the group of units of $A_\infty$. 
 
Let $\ki=k((t^{-1}))$ be the $v_\infty$-adic completion of $K$. The valuation $v_\infty$ extends in an obvious way to $\ki$. Let $\hat{A}_\infty\subset \ki$ be the valuation ring of $v_\infty$, and $\hat{\m}_\infty$ its maximal ideal.

On $\ki$ we may consider the \emph{degree function} $|~|:=-v_\infty$, which is an extension of the ordinary degree of polynomials: $|a|=\deg a$ for all $a\in A$.

Although for many applications it is sufficient to deal only with lattices over the polynomial ring $A$, we consider a more general situation.

Consider a principal ideal domain $R$ with field of fractions $K_R\subset K_\infty$. Typical instances for $R$ will be $R=A,\ A_\infty $, or $\hat{A}_\infty$.
\begin{definition}\label{norm} A \emph{norm}, or \emph{length function} on an $R$-module $L$ is a mapping
$$
\v{~}\colon L\lra \{-\infty\}\cup\R
$$ 
satisfying the following conditions:
\begin{enumerate}
\item $\Vert x+y\Vert\le \max\{\v{x},\Vert y\Vert\}$, for all $x,y\in L$, \item $\Vert ax\Vert=|a|+\v{x}$,
for all $a\in R$, $x\in L$,
\item $\v{x}=-\infty$ if and only if $x=0$.
\end{enumerate}
\end{definition}

For $r\in\R$ we define 
$$L_{\leq r}:=\{x\in L\mid \v{x}\leq r\},\qquad L_{< r}:=\{x\in L\mid \v{x}< r\}.$$

Note that  for any $x_1,x_2\in L$ with $\v{x_1}\neq \v{x_2}$, it holds
\begin{align}\label{ineq}
\v{x_1+x_2}=\max\{\v{x_1},\v{x_2}\}.
\end{align}

 Clearly, the degree function itself $|~|$ is a norm on $R$.

Let $e>1$ be a real number. By using $e^{|~|}$ instead of $|~|$, and $e^{\v{~}}$ instead of $\v{~}$, we would get the usual properties of a norm: $\v0=0$, $\v{ax}=|a|\v{x}$. However, we prefer to use additive length functions because then $|a|\in\Z$ is the ordinary degree of $a$, for any $a\in \ki$. 
Another psychologically disturbing consequence of our choice is the fact that a lattice may have negative volume (cf. Section \ref{detod}).

\begin{definition}\label{deflat}
Let $L$ be a finitely generated $R$-module and $\v{~}$ a norm on $L$. 
The pair $(L,\v{~})$ is said to be a \emph{lattice} over $R$ if $\dim_k L_{\leq r}<\infty$ for all $r\in \R$.

A \emph{normed space over $K_R$} is a pair $(E,\v{~})$, where $E$ is a finite dimensional $K_R$-vector space equipped with a norm $\v{~}$, admitting a finitely generated $R$-submodule $L\subset E$ of full rank such that $(L,\v{~})$ is a lattice.
\end{definition}

Clearly, if $(L,\v{~})$ is a lattice, then $L\otimes_RK_R$ is a normed space, with
the norm function obtained by extending $\v{~}$ in an obvious way. The second property in Definition \ref{norm} of
a norm shows that $L$ has no $R$-torsion, so that $L$ is a free
$R$-module and it is embedded into the normed space $L\otimes_RK_R$. 

Conversely, if $(E,\v{~})$ is a normed space, then any finitely generated $R$-submodule $M$ of full rank is a lattice with the norm obtained by restricting $\v{~}$ to $M$.

In fact, let $L\subset E$ be a sub-$R$-submodule such that  $(L,\v{~})$ is a lattice. Since there exists an $a\in K_R\setminus\{0\}$ with $a M\subset L$, we obtain 
$$
\dim_k M_{\leq r}=\dim_k (a M)_{\leq r+|a|}\le \dim_k (L)_{\leq r+|a|}<\infty,\quad \text{for all }r\in \R.
$$

\noindent{\bf Examples. }

The lattice $\oo$ is by definition the pair $(A,|~|)$, where $|~|$ is the degree function. Analogously, we define the normed space $\kk=(K,|~|)$.

Let $F/k$ be an algebraic function field and denote $\pp_\infty(F)$ the set of places over $P_\infty$ of $F$. Then, 
$$
\v{~}:=\-\min_{P\in\pp_\infty(F)}\Big\{\frac{-v_P(~)}{e(P/P_\infty)}\Big\}
$$
is a norm on $F$ and $(F,\v{~})$ becomes a normed space over $K$ (cf. Section \ref{RRChapter}).\medskip

Many concepts can be introduced both for lattices
and normed spaces. By the above considerations it is easy to deduce one from
each other. In the sequel we give several definitions for lattices over $A$ and we leave
to the reader the formulation of similar concepts for more general lattices or normed spaces.

\begin{definition}
A lattice homomorphism between the lattices $(L,\v{~})$ and $(L',\v{~}')$
is an $A$-module homomorphism $ \varphi\colon L\lra L'$ such that
$\v{\varphi(x)}'=\v{x}$ for all $x\in L$. 

A lattice isomorphism is called an \emph{isometry} between $(L,\v{~})$ and $(L',\v{~}')$.
\end{definition}

\begin{definition}
The \emph{orthogonal sum} of two lattices $(L,\vv)$, $(L',\vv')$ is
defined as:
$$
L\perp L'=(L\oplus L',\vv), \quad \v{(x,x')}=\max\{\v{x},\v{x'}'\},
$$
for all $x\in L$, $x'\in L'$. Instead of $\perp_{i=1}^nL$ we write for simplicity $L^n$.
\end{definition}

\begin{definition}\label{twist} 
Given a lattice $\ll=(L,\v{~})$ and a real number $r$, we define the \emph{twisted lattice} $\ll(r)$ to be the pair $(L,\v{~}')$, where $\v{~}'=\v{~}+r$.
\end{definition}

\begin{lemma}\label{ri}
Let $(L,\v{~})$ be an $A$-lattice of rank $n$. For $1\leq i\leq n$, consider
$$
\mathcal{R}_i=\{\max\{\v{x_1},\dots,\v{x_i}\}\mid x_1,\dots,x_i\in L \text{ are } A\text{-linearly independent }\}.
$$ 
Then, $r_i:=\inf(\mathcal{R}_i)$ exists and is attained by some vector in $L$.
These numbers $r_1\le\cdots \le r_n$ are called the \emph{successive minima} of $L$.
\end{lemma}
\begin{proof}
Suppose $\lambda_1>\lambda_2>\dots$ is a strictly decreasing sequence in $\mathcal{R}_i$. Then, we obtain a chain of $k$-vector spaces
$$
L_{\leq \lambda_1}\supsetneq L_{\leq \lambda_2}\supsetneq \dots.
$$
This is a contradiction to the fact that $L$ is a lattice.
\end{proof}

\subsection{Reduced bases}\label{subredbasis}
We fix throughout this section a normed space $(E,\v{~})$ over $K$ of dimension $n$. By a basis of $E$ we mean a $K$-basis. By a basis of a lattice $L\subset E$ we mean an $A$-basis.

\begin{definition}\label{defred1}
Let $\bb=\{b_1,\dots,b_m\}$ be a subset of $E\setminus\{0\}$. We say that $\bb$ is \emph{reduced} if for all $a_1,\dots,a_m\in K$, it holds
\begin{equation}\label{defred}
\v{a_1b_1+\cdots +a_mb_m}=\max\limits_{1\leq i\leq m}\{\v{a_ib_i}\}. 
\end{equation}
Equivalently, it suffices to check (\ref{defred}) for all families $a_1,\dots,a_m\in A$.
\end{definition}

The following observations are an immediate consequence of the definition of reduceness.

\begin{lemma}\label{product} \quad
\begin{enumerate}
\item A reduced family is $K$-linearly independent.
\item Let $\bb=\{b_1,\dots,b_m\}\subset E$ be a reduced set. Then, for any $a_1,\dots,a_m\in K^*$, the set $\{a_1b_1,\dots,a_mb_m\}$ is reduced.

\end{enumerate}
\end{lemma}
For a basis $\bb=(b_1,\dots,b_n)\in E^n$, denote by $c_\bb:E\rightarrow K^n$ the $K$-isomorphism mapping $x\in E$ to its coordinates in $K^n$ with respect to the basis $\bb$. 

\begin{lemma}\label{cbisiso}
Let $\bb=(b_1,\dots,b_n)\in E^n$ be a basis of $E$ with the vectors ordered by increasing length: 
$$
r_1:=\v{b_1}\le\cdots\le r_n:=\v{b_n}.
$$
Then, the following conditions are equivalent:
\begin{enumerate}
\item $\bb$ is a reduced basis of $E$.
\item $c_\bb:E\rightarrow \kk(r_1)\perp\cdots\perp\kk(r_n)$ is an isometry.
\item The lattice $L=\gen{\bb}_A$ is isometric to $\oo(r_1)\perp\cdots\perp\oo(r_n)$.
\end{enumerate}
\end{lemma}
\begin{proof}
The fact that $c_\bb$ is an isometry is a reformulation of Definition \ref{defred1}. Also, the fact that the $A$-isomorphism $L\simeq A^n$ obtained by restricting $c_\bb$ to $L$ is an isometry between $L$ and $\oo(r_1)\perp\cdots\perp\oo(r_n)$ is a reformulation of Definition \ref{defred1} too. 
\end{proof}

\begin{proposition}\label{basic}
Let $\bb=(b_1,\dots,b_n)\in E^n$ be a reduced basis of $E$ with 
$$
r_1:=\v{b_1}\le\cdots\le r_n:=\v{b_n}.
$$ 
Let $L=\gen{\bb}_A$ be the lattice generated by $\bb$. Then,
\begin{enumerate}
\item $\v{E}:=\{\v{x}\mid x\in E\setminus\{0\}\}=(r_1+\Z)\cup\cdots\cup (r_n+\Z)$.

This set induces a finite subset of $\R/\Z$ called the \emph{signature} of $E$:
$$
\mathrm{Sig}(E):=\v{E}/\Z=\left\{r_1+\Z,\dots,r_n+\Z\right\}\subset \R/\Z.
$$

\item $r_1\le\cdots\le r_n$ are the successive minima of $L$.
\item For any $r\in \R$, the following family is a $k$-basis of $L_{\leq r}$:
$$
\{b_it^{j_i}\mid 1\leq i\leq n,\quad 0\leq j_i\leq \lfloor r-r_i\rfloor\}.
$$
In particular, take $r_0=-\infty$, $r_{n+1}=\infty$ and let $0\le \kappa\le n$ be the index for which $r_\kappa\le r<r_{\kappa+1}$. Then,
$$
\dim_k L_{\le r}=\sum_{i=1}^\kappa(\lfloor r-r_i\rfloor+1).
$$
\end{enumerate}
\end{proposition}

\begin{proof} The length of any nonzero vector $x=\sum_{i=1}^n a_ib_i\in E$ is of the form
$$
\v{x}=\max_{1\leq i\leq n}\{\v{a_ib_i}\}=|a_j|+\v{b_j}=|a_j|+r_j\in r_j+\Z.
$$
This proves the first item.

For any $1\leq j\leq n$, the vectors $x=\sum_{i=1}^n a_ib_i\in L$ satisfying 
 $$
 \v{b_j}>\v{x}=\max_{1\leq i\leq n}\{\v{a_ib_i}\}
 $$
lie necessarily in the submodule $\gen{b_1,\dots,b_{j-1}}_A$. Hence, for any $A$-linearly independent family $x_1,\dots,x_j\in L$, we know that
$$
\max\{\v{x_1},\dots,\v{x_j}\}\geq \v{b_j}.
$$ 
This proves the second item.

For the last statement, the element $x=\sum_{i=1}^n a_ib_i$ belongs to $L_{\leq r}$ if and only if
$$
\v{x}=\max_{1\leq i\leq n}\{\v{a_ib_i}\}\leq r.
$$
This is equivalent to $a_{\kappa+1}=\cdots=a_{n}=0$ and 
$$
|a_i|\leq r-r_i,\qquad 1\leq i\leq \kappa.
$$
The subset of all polynomials $a\in A$ satisfying $|a|\leq r-r_i$ is a $k$-vector subspace with basis $1,t,\dots,t^{\lfloor r-r_i\rfloor}$. This ends the proof of the last item.
\end{proof}

The following observation is a direct consequence of item (\textit{1}) of Proposition \ref{basic}.

\begin{corollary}\label{discreteintervall}
For any real numbers $r<s$, the set $\v{E}\cap [r,s]$ is finite.
\end{corollary}

The most relevant property of a reduced basis is that the lengths of the vectors attain the successive minima of the lattice generated by the basis. Actually, this property characterizes reduced bases. 

\begin{theorem}\label{sucmintored}
Let $(L,\v{~})$ be a lattice and $r_1\le\cdots \le r_n$ its successive minima. Let $\bb=(b_1,\dots,b_n)$ be a family of $A$-linearly independent elements in $L$ such that $\v{b_i}=r_i$, for all $1\le i\le n$. Then, $\bb$ is a reduced basis of $L$.
\end{theorem}

\begin{proof}
Let us first show that $\bb$ is reduced by induction on $n$. Reduceness being obvious for $n=1$, assume the statement holds for lattices of rank $n-1$.

Take $a_1,\dots,a_n\in A$ and set $u=a_1b_1+\cdots+a_{n-1}b_{n-1}$. We want to show that
$$
\v{u+a_nb_n} =\max\{\v{u},\v{a_nb_n}\},
$$
since by induction hypothesis it holds $\v{u}=\max_{1\leq i<n}\{\v{a_ib_i}\}$. For $\v{u}\neq\v{a_nb_n}$ the statement follows from (\ref{ineq}).

Suppose $\v{u}=\v{a_nb_n}$ and $\v{u+a_nb_n} <\max\{\v{u},\v{a_nb_n}\}=\v{u}$. In particular, we have $a_n\neq 0$. We fix $I:=\{1\leq j\le n\mid \v{a_ib_i}=\v{u}\}$; note that $n\in I$ by our assumption. For $i\in I$ we write $a_i=\lambda_it^{d_i}+a'_i$, where  $\lambda_i\in k$ and $d_i=|a_i|>|a'_i|$. Then, if we take $u_0=\sum_{i\in I}\lambda_it^{d_i}b_i$, it holds
$u+a_nb_n=u_0+u'$ with $u'\in L$ having $\v{u'}<\v{u}$. Hence,
\begin{equation}\label{form}
\v{u_0}\le \max\{\v{u+a_nb_n},\v{u'}\}<\v{u}=|a_i|+r_i=d_i+r_i,\  \forall i\in I.
\end{equation}
Since $r_i\leq r_n$, we have $d_n\leq d_i$ for $i \in I$. By (\ref{form}), the element $b=t^{-d_n}u_0$ belongs to $L$ and has $\v{b}<r_n$. Since $b_1,\dots b_{n-1},b\in L$ are linearly independent, this contradicts the minimality of $r_n$. This proves that $\bb$ is reduced.

Finally, let us show that $\bb$ generates $L$. 
Assume there exists an element $b\in L$ with $b\notin\gen{\bb}_A$. Since $\bb$ is a $K$-basis of $E$, we obtain $b=\sum_{i=1}^na_ib_i$ with at least one $a_i\in K\setminus A$. We set $I=\{1\leq i\leq n\mid a_i\notin A\}$ and consider
\begin{align}\label{equ}
\sum_{i\in I}a_ib_i=b-\sum_{i\not\in I}a_ib_i\in L. 
\end{align}
As the set $\bb$ is reduced, it holds 
$$
\V{\sum_{i\in I}a_ib_i}=\max_{i\in I}\{\v{a_i b_i}\}=\v{a_jb_j},
$$
 for some $j\in I$. If $|a_j|\geq0$ we can write $a_j=a+a'_j$ with $a\in A$ and $a'_j\in \m_\infty$ and subtract $ab_j$ in (\ref{equ}) from both sides. Therefore, we can assume that $|a_j|<0$ and get $\v{{\sum_{i\in I}a_ib_i}}<\v{b_j}=r_j$. By setting $b'_j={\sum_{i\in I}a_ib_i}$, we obtain the set $\{b_1,\dots,b_{j-1},b'_j,b_{j+1},\dots,b_n\}$ of $A$-linearly independent elements in $L$. This is in contradiction with the minimality of $\v{b_j}=r_j$.
\end{proof}

From Lemma \ref{ri}, it follows easily the existence of linearly independent elements in a given lattice, whose length attains the successive minima. By Theorem \ref{sucmintored}, this guarantees the existence of reduced bases in any normed space.

\begin{corollary}\label{exist}
Every lattice admits a reduced basis.
\end{corollary}

\subsection{Reduceness criteria} \label{secredcri}
In this section we define a reduction map, which leads to a practical criterion to check wether a basis in a normed space $(E,\vv)$ is reduced or not. For any $r\in\R$ the subspaces $E_{\le r}\supset E_{<r}$
are $A_\infty$-submodules of $E$ such that $\m_\infty E_{\le r}\subset E_{<r}$. Their quotient, $$V_r:=E_{\le r}/E_{<r}$$ is a $k$-vector space, admitting a kind of \emph{reduction map}:
$$
\red_r\colon E_{\le r}\lra V_r,\quad x\mapsto x+E_{<r}.
$$ 
Clearly, $V_r$ is nonzero if and only if $r\in\v{E}$.  

\begin{definition}\label{subset}
For any $\bb\subset E\setminus\{0\}$ and $\rho\in\R/\Z$, we denote 
$$
\bb_\rho:=\{b\in\bb\mid \v{b}+\Z=\rho\}.
$$
\end{definition}

\begin{lemma}\label{pieces} 
Let $\bb=(b_1,\dots,b_n)$ be a basis of $E$, and let $\bb=\bigcup_{\rho\in\R/\Z}\bb_\rho$ be the partition determined by classifying all vectors in $\bb$ according to its length modulo $\Z$. Then, $\bb$ is reduced if and only if all subsets $\bb_\rho$ are reduced.
\end{lemma}

\begin{proof}
Any subset of a reduced family is reduced. Thus, we need only to show that $\bb$ is reduced if all $\bb_\rho$ are reduced.

Let $I=\{\rho\in\R/\Z\mid \bb_\rho\ne\emptyset\}$.
We have $E=\bigoplus_{\rho \in I}E_\rho$, where $E_\rho$ is the subspace of $E$ generated by $\bb_\rho$.
Take $a_1,\dots,a_n\in K$ and let $x=\sum_{i=1}^n a_ib_i$. This element splits as $x=\sum_{\rho\in I} x_\rho$, where $x_\rho=\sum_{b_i\in \bb_\rho}a_ib_i$. 
Since all values $\v{x_\rho}$ are different (because $\v{a_ib_i}\equiv \v{b_i}\bmod{\Z}$),
we have $\v{x}=\max_{\rho\in I}\{\v{x_\rho}\}$. On the other hand, since all $\bb_\rho$ are reduced, we have $\v{x_\rho}=\max_{b_i\in \bb_\rho}\{\v{a_ib_i}\}$. Thus, $\bb$ is reduced.    
\end{proof}

\begin{definition}\label{zc}
For $a\in A_\infty$, consider the series expansion of $a$ with respect to the local parameter $t^{-1}$ at $P_\infty$:
$$
a=\sum_{i=0}^\infty\lambda_i\,t^{-i},\quad \lambda_i\in k.
$$
We define the \emph{zero coefficient} of $a$ as $\zc(a)=\lambda_0\in k$. It is uniquely determined by the condition $|a-\zc(a)|<0$. Clearly, $\zc(a)=0$ if and only if $a\in\m_\infty$.
\end{definition}

The next result is inspired by a criterion of W.M. Schmidt \cite{Sch,Schoe}, which was developed in the context of Puiseux expansions of functions in function fields.

\begin{theorem}\label{SS} 
Let $\bb$ be a basis of $E$, and let $\bb=\bigcup_{\rho\in \R/\Z}\bb_\rho$ be the partition determined by classifying all vectors in $\bb$ according to its length modulo $\Z$. For each $\bb_\rho\ne\emptyset$, choose a real number $r\in\rho$, and write
$$
\v{b}=r-m_b,\quad m_b\in\Z,\quad \mbox{for all }b\in \bb_\rho.
$$
Then, $\bb$ is reduced if and only if the elements $\{\red_r(t^{m_b}b)\mid b\in\bb_\rho\}\subset V_r$ are $k$-linearly independent for all $\bb_\rho\ne\emptyset$.
\end{theorem}

\begin{proof}
By Lemma \ref{pieces} we can assume that all elements in $\bb$ have the same length modulo $\Z$. Thus, $I=\{\rho\}$ contains a single element and $\v{t^{m_b}b}=r$ for all $b\in \bb$.

By Lemma \ref{product}, $\bb$ is reduced if and only if $\{t^{m_b}b\mid b\in\bb\}$ is reduced. Thus, we may assume that $\v{b}=r$ for all $b\in \bb$.

Let $(a_b)_{b\in\bb}$ a family of elements in $K$, not all of them equal to zero. By multiplying these elements by the same (adequate) power of $t$ we may assume that $\max\{|a_b|\}=0$, so that $\max\{\v{a_bb}\}=r$. Let $\cc=\{b\in\bb\mid |a_b|=0\}$. Clearly,
\begin{align*}
\V{\sum_{b\in\bb}a_bb}=r \sii & 
\V{\sum_{b\in\cc}a_bb}=r \sii
\V{\sum_{b\in\cc}\zc(a_b)b}=r \\\sii &
\red_r\left(\sum_{b\in\cc}\zc(a_b)b\right)\ne0\sii
\sum_{b\in\cc}\zc(a_b)\red_r\left(b\right)\ne0.
\end{align*}
Hence, the condition (\ref{defred}) of reduceness, for all families $(a_b)_{b\in\bb}$ in $K$, is equivalent to $\{\red_r(b)\mid b\in\bb\}$ being $k$-linearly independent.  
\end{proof}

\begin{corollary}\label{basisV} 
With the above notation, $(\red_r(t^{m_b}b)\mid b\in\bb_\rho)$ is a $k$-basis of $V_r$. In particular, it holds $\dim_k V_r=\#\bb_\rho$.
\end{corollary}

\begin{proof}
By the previous theorem, this family is $k$-linearly independent. Let us show that it generates $V_r$ as well. As in the proof of the theorem, we may assume that $\v{b}=r$ for all $b\in\bb_\rho$. 

Suppose $x\in E$ has $\v{x}=r$, and write it as $x=\sum_{b\in\bb}a_bb$, for some $a_b\in K$. By reduceness, we deduce
$$
r=\v{x}=\max_{b\in\bb}\{\v{a_bb}\}=\max_{b\in\bb_\rho}\{\v{a_bb}\}=r+\max_{b\in\bb_\rho}\{|a_b|\}.
$$
Hence, $|a_b|\leq 0$ for all $b\in \bb_\rho$ and $\cc:=\{b\in\bb_\rho\mid |a_b|=0\}\neq \emptyset$. Clearly,
$x\in \sum_{b\in\cc}\zc(a_b)b+E_{<r}$, and $\red_r(x)$ is a $k$-linear combination of $\{\red_r(b)\mid b\in\cc\}$.
\end{proof}

\noindent{ \bf Notation. }
Let $L$ be a lattice of rank $n$, and $r_1\le\cdots \le r_n$ its successive minima. We denote by
$$
\sm(L)=(r_1,\dots,r_n),\quad \smb(L)=\{r_1+\Z,\dots,r_n+\Z\},
$$
the vector of successive minima of $L$ and the multiset formed by their classes in $\R/\Z$, with due count of multiplicities. 

\begin{lemma}\label{multi}
All lattices $L\subset E$ in a normed space have the same multiset $\smb(L)$. 
We denote by $\sm(E)$ this common multiset
\end{lemma}

\begin{proof}
By Corollary \ref{exist} there exists a reduced basis $\bb=(b_1,\dots,b_n)$ of any lattice $L$ in $E$. By Proposition \ref{basic}, the underlying set of $\smb(L)$ is the signature of $E$, which depends only on $E$. Finally, for each $\rho\in\R/\Z$, the multiplicity of $\rho$ as an element of the multiset $\smb(L)$ is the cardinality of the set $\bb_\rho$. By Corollary \ref{basisV} this multiplicity $\#\bb_\rho$ is also independent of $L$.
\end{proof}

\begin{corollary}\label{classification}
Two lattices $L$, $L'$ are isometric if and only if $\sm(L)=\sm(L')$.

Two normed spaces $E$, $E'$ are isometric if and only if $\sm(E)=\sm(E')$.
\end{corollary}

\begin{proof}
From $\aut_A(A)=k^*$ and $\aut_K(K)=K^*$, we deduce immediately:
\begin{align*}
\oo(r)\mbox{ isometric to } \oo(r')&\Longleftrightarrow r=r'\\
\kk(r)\mbox{ isometric to } \kk(r')&\Longleftrightarrow r+\Z=r'+\Z,
\end{align*}
for any given real numbers $r,r'\in\R$. The corollary follows from the existence of reduced bases and Lemma \ref{cbisiso}.
\end{proof}

\subsection{Orthonormal bases and isometry group}

\begin{definition}
Let $E$ be a normed space and $\bb=(b_1,\dots,b_n)$ a reduced basis of $E$. 
We say that $\bb$ is \emph{orthonormal} if $-1<\v{b_1}\leq \dots\leq\v{b_n}\leq 0$.
\end{definition}

Clearly, two orthonormal bases of the same normed space $E$ have the same multiset of lengths of their vectors. 

The aim of this section is to describe maps between normed spaces. In particular, we want to derive properties of the transition matrices between orthonormal bases. 
\begin{definition}
The set of all isometries $(E,\v{~})\rightarrow (E,\v{~})$ is denoted by\linebreak $\mathrm{Aut}(E,\v{~})$. This set has a natural group structure. We call it the \emph{isometry group} of the normed space $(E,\v{~})$.
\end{definition}

\begin{lemma}\label{mapsred}
Every morphism of normed spaces is injective and maps a reduced set to a reduced one.
\end{lemma}

\begin{proof}
A length-preserving map is injective because it must have a trivial kernel. Also, it clearly preserves condition (\ref{defred}) from Definition \ref{defred1}. 
\end{proof}

\begin{lemma}\label{equivMap}
Let $(E,\v{~})$ and $(E',\v{~}')$ be normed spaces with $\sm(E)=\sm(E')$ and $\varphi:E\rightarrow E'$ be a $K$-linear map. Then, the following statements are equivalent:
\begin{enumerate}
\item The map $\varphi$ is an isometry.
\item The map $\varphi$ sends orthonormal bases of $E$ to orthonormal bases of $E'$.
\item The map $\varphi$ sends a fixed orthonormal basis of $E$ to an orthonormal basis of $E'$.
\end{enumerate}
\end{lemma}

\begin{proof}
The first statement implies the second one by Lemma \ref{mapsred}, and the second one implies trivially the third one.

We show that that that (3) implies (1). Since $\varphi$ maps an orthonormal basis $\bb=(b_1,\dots,b_n)$ of $E$ to an orthonormal basis $\varphi(\bb)$ of $E'$, the $K$-linear map $\varphi$ is an isomorphism. As $\sm(E)=\sm(E')$, the two sequences of lengths
$$
-1<\v{b_1}\leq \dots \leq \v{b_n}\leq 0,\quad -1<\v{\varphi(b_1)}'\leq \dots \leq \v{\varphi(b_n)}'\leq 0
$$
coincide. Therefore, for any $x\in E$ with $x=\sum_{i=1}^na_ib_i$, we have
$$
\v{x}=\max_{1\leq i\leq n}\{\v{a_i b_i}\}=\max_{1\leq i\leq n}\{\v{a_i \varphi(b_i)}'\}=\V{\sum_{i=1}^na_i\varphi(b_i)}'=\v{\varphi(x)}',
$$ 
so that $\varphi$ preserves lengths.
\end{proof}

\begin{theorem}\label{autgl}
For $r\in \R$ it holds $\mathrm{Aut}( \kk^n(r))=\gl_n(A_\infty)$. 
\end{theorem}

\begin{proof}
Let us first take $r=0$. Recall that $\kk^n=(K^n,\v{~})$ with $\v{(a_1,\dots,a_n)}=\max_{1\leq i\leq n}\{|a_i|\}$. Let $(e_1,\dots, e_n)$ be the standard basis of $K^n$, which is an orthonormal basis of $\kk^n$. By Lemma \ref{equivMap}, for $T\in \gl_n(K)$ the map $T:K^n\rightarrow K^n$ is an isometry if and only if $(e_1T,\dots, e_nT)$ is an orthonormal basis. In particular, the rows of $T$ have length $0$, so that $T\in \gl_n(K)\cap A_\infty^{n\times n}$. By Theorem \ref{SS}, the rows of $T$ are reduced if and only if they  are linearly independent $\bmod\ \m_\infty^n$. Clearly, this holds if and only if $\det T \notin \m_\infty$. Thus, $T$ is an isometry if and only if $T\in\gl_n(A_\infty)$.

The general case follows from the same argument, having in mind that an orthonormal basis of $\kk^n(r)$ is $t^{-\lceil r\rceil}\bb$, where $\bb$ is the standard basis of $K^n$.
\end{proof}

\begin{definition}\label{oldgroup}
Let $n=m_1+\cdots+m_\kappa$ be a partition of a positive integer $n$ into a sum of positive integers. Let $T$ be an $n\times n$ matrix with entries in $A_\infty$. The partition of $n$ determines a decomposition of $T$ into blocks:
$$
T=(T_{ij}),\quad T_{ij}\in A_\infty^{m_i\times m_j},\ 1\le i,j\le \kappa.
$$  

The \emph{orthonormal group} $O(m_1,\dots,m_\kappa,A_\infty)$ is the subgroup of $\gl_n(A_\infty)$ formed by all $T\in A_\infty^{n\times n}$, which satisfy the following two conditions:
\begin{enumerate}
\item $T_{ii}\in\gl_{m_i}(A_\infty)$, for all $1\le i\le \kappa$.
\item $T_{ij}\in \m_\infty^{m_i\times m_j}$, for all $j> i$.
\end{enumerate}
\end{definition}

\begin{theorem}\label{helpi}
Let $-1<r_1<\dots < r_\kappa\leq 0$ be a sequence of real numbers. Then, for $m_1,\dots,m_\kappa\in\Z_{>0}$ it holds
$$
\mathrm{Aut}( \bot_{i=1}^\kappa \kk^{m_i}(r_i))=O(m_1,\dots,m_\kappa,A_\infty).
$$
\end{theorem}

\begin{proof}
For $1 \le i\le \kappa$, let $n_i=m_1+\cdots+m_i$ and $n=n_\kappa$.
Let $E= \bot_{i=1}^\kappa \kk^{m_i}(r_i)$ and denote by $\v{~}$ the norm on $E$; that is, 
\begin{align}\label{standynorm}
\v{(a_1,\dots, a_n)}=\max\left\{|a_j|+r_{i_j}\mid 1\le j\le n,\ n_{i_j}<j\le n_{i_j+1}\right\}.
\end{align}

Let $e_1,\dots,e_n$ be the standard basis of $K^n$, which is an orthonormal basis of $E$. By Lemma \ref{equivMap}, $\mathrm{Aut}(E)$ consists of the matrices $T\in\gl_n(K)$ whose rows form an orthonormal basis of $E$. 
Let us show that this property characterizes the matrices in $O(m_1,\dots,m_\kappa,A_\infty)$. To this end, we will use the following:\medskip

\noindent{\textbf{Claim}: }
Let $i\in\{1,\dots, \kappa\}$ and $b_1,\dots,b_{m_i}\in E$. It holds $\v{b_1}=\cdots=\v{b_{m_i}}=r_i$ and $\red_{r_i}(b_1),\dots,\red_{r_i}(b_{m_i})$ are $k$-linearly independent if and only if the following two conditions are satisfied:
\begin{enumerate}
\item $b_{l}=(b_{1,l},\dots,b_{n,l})\in A^{n_i}_\infty\times\m^{n-n_i}_\infty$, for all $1\leq l\leq m_i$.\label{impro1}
\item $Q:=(b_{j,l}\mid 1\leq l \leq  m_i,\,n_{i-1}< j \leq n_i )\in \gl_{m_i}(A_\infty)$.
\end{enumerate}\medskip

The statement of the theorem follows immediately from the claim. In fact, for any $T\in\gl_n(K)$, Theorem \ref{SS} shows that the rows of $T$ form an orthonormal basis of $E$ if and only if the $\kappa$ subfamilies of the set of rows determined by the partition $n=m_1+\cdots+m_\kappa$ satisfy the condition of the claim. By the claim this is equivalent to $T\in O(m_1,\dots,m_\kappa,A_\infty)$.

We have to prove the claim. By (\ref{standynorm}), $\v{b_l}\leq r_i$, for $1\leq l\leq m_i$, is equivalent to item (\ref{impro1}) of the claim, since $-1<r_1< \dots< r_\kappa\leq 0$. 

Note that $b_{j,l}e_j\in E_{<r_i}$, for all $1\leq j\leq n_{i-1}$ (because $|b_{j,l}|\leq 0$, $\v{e_j}<r_i$) and for all $n_i<j\leq n$ (because $|b_{j,l}|\leq -1$, $\v{e_j}<r_i+1$). Thus,
$$
b_l=\sum_{j=1}^nb_{j,l}e_j\in \sum_{j=n_{i-1}+1}^{n_i}\zc(b_{j,l}) e_j+E_{<r_i},\quad 1\le l\le m_i.
$$ 

Clearly, $\red_{r_i}(b_1),\dots,\red_{r_i}(b_{m_i})$ are $k$-linearly independent if and only if the matrix $(\zc(b_{j,l })_{1\leq l \leq  m_i,\,n_{i-1}< j \leq n_i })$ belongs to $\gl_{m_i}(k)$. This is equivalent to $Q\in \gl_{m_i}(A_\infty)$ and $\v{b_l}= r_i$ for $1\leq l\leq m_i$. This ends the proof of the claim. 
\end{proof}
Since every normed space $(E,\vv)$ is isometric to some $\bot_{i=1}^\kappa \kk^{m_i}(r_i)$ (Lemma \ref{cbisiso}), Theorem \ref{helpi} reveals the general structure of $\aut(E,\vv)$.

Let $\bb=(b_1,\dots,b_n)$ and $\bb'=(b'_1,\dots,b'_n)$ be two bases of $E$. The \emph{transition matrix} from $\bb$ to $\bb'$ is the unique matrix $T=T(\bb\rightarrow \bb')\in \gl_n(K)$ such that
$$T(b'_1\dots b'_n)^\tp=(b_1\dots b_n)^\tp.$$ 

Thus, if $(a_1,\dots, a_n)$ are the coordinates of a vector $u$ in $E$ with respect to the basis $\bb$, then $(a_1\dots a_n)T$ is the coordinate vector of $u$ with respect to the basis $\bb'$.

\begin{lemma}\label{transition} 
Let $\bb'$ be an orthonormal basis of $E$ and let $m_1,\dots,m_\kappa$ be the multiplicities of the lengths of the vectors of $\bb'$. Then, a basis $\bb$ of $E$ is orthonormal if and only if the transition matrix from $\bb$ to $\bb'$ belongs to $O(m_1,\dots,m_\kappa,A_\infty)$.  
\end{lemma}

\begin{proof}
Let $E':= \bot_{i=1}^\kappa \kk^{m_i}(r_i)$, where $r_1,\dots, r_\kappa$ are the pairwise different lengths of the vectors in $\bb'$. The transition matrix $T$ from $\bb$ to $\bb'$ determines a $K$-isomorphism $T:E'\rightarrow E'$ fitting into the following commutative diagram:
$$\begin{tikzcd}E \arrow{rd}[swap]{c_{\bb'}} \arrow{r}{c_{\bb}}  &E'\arrow{d}{T} \\                           & E' \end{tikzcd}$$

By Lemma \ref{cbisiso}, $c_{\bb'}$ is an isometry. Hence, $T$ is an isometry if and only if $c_{\bb}$ is an isometry. By Theorem \ref{helpi}, $T$ is an isometry if and only if $T\in O(m_1,\dots,m_\kappa,A_\infty)$. By Lemma \ref{equivMap}, $c_\bb$ is an isometry if and only if $\bb$ is an orthonormal basis. 
\end{proof}

\subsection{Determinant and orthogonal defect}\label{detod}

\begin{definition}[Volume]
Let $\bb$ be a basis of a normed space $E$. 
We define the \emph{volume} of $\bb$ as $\vol(\bb):=\sum_{b\in\bb}\v{b}$.

We define the \emph{volume} of $E$ as the volume of any orthonormal basis of $E$. 
The \emph{volume} of a lattice $L$ is defined to be the volume of a reduced basis of $L$.
We use the notation $\vol(E)$ and $\vol(L)$, respectively.
\end{definition}

\begin {definition}[Determinant]
Let $\bb$ be a basis of a normed space $E$. We define the \emph{determinant} $d(\bb)$ of $\bb$ to be the fractional ideal of $A$ generated by the determinant of the transition matrix from $\bb$ to an orthonormal basis of $E$.

The \emph{determinant} $d(L)$ of a lattice $L$ is defined to be the determinant of any basis of $L$. 
\end {definition}

By Lemma \ref{transition} the definition of the determinant is independent of the choice of the orthonormal basis of $E$.

For $h\in K$, we set $|hA|:=|h|$ in order to extend the degree function $|~|$ to fractional ideals of $A$.

\begin{lemma}[Hadamard's inequality]
Let $\bb$ be a basis of $E$. Then,
$$
|d(\bb)|\le\vol(\bb)-\vol(E).
$$
\end{lemma}

\begin{proof}Let $\bb=(b_1,\dots,b_n)$ and let $\bb'=(b'_1,\dots,b'_n)$ be an orthonormal basis of $E$. Let $T=(t_{i,j})$ be the transition matrix from $\bb$ to $\bb'$. Since $\bb'$ is reduced, for every $1\le i,j\le n$, we have
$$
\v{t_{j,i}b'_i}\le \max_{1\le k\le n}\{\v{t_{j,k}b'_k}\}=\v{b_j}.
$$
Hence, every summand of $\det T$, corresponding to a permutation $\tau$ of the set $\{1,\dots,n\}$, has degree: 
\begin{align*}
|t_{1,\tau(1)}\cdots t_{n,\tau(n)}|=&\,
|t_{1,\tau(1)}|+\cdots +|t_{n,\tau(n)}|\\\le &\, 
\v{b_1}-\v{b'_{\tau(1)}}+\cdots +\v{b_n}-\v{b'_{\tau(n)}}\\=&\,\vol(\bb)-\vol(E). 
\end{align*}
Thus, $|\det T|\le\vol(\bb)-\vol(E)$.
\end{proof}

\begin{definition}[Orthogonal defect]
The difference $$\od(\bb):=\vol(\bb)-\vol(E)-|d(\bb)|\ge0$$ is called the \emph{orthogonal defect} of $\bb$.
\end{definition}
If $\bb$ is orthonormal, then $\vol(B)=\vol(E)$ and $|d(\bb)|=0$, so that $\od(\bb)=0$.
\begin{lemma}\label{OD} 
Let $\bb=(b_1,\dots,b_n)$ be a basis of $E$. Then, for any element $x=\sum_{i=1}^na_ib_i\in E$, we have
\begin{equation}\label{od}
\v{a_ib_i}\le \v{x}+\od(\bb), \text{ for all }\,1\le i\le n.
\end{equation}
\end{lemma}

\begin{proof}Let $\bb'=(b'_1,\dots,b'_n)$ be an orthonormal basis of $E$, and $T$ the transition matrix from $ \bb$ to $\bb'$. We have $x=\sum_{i=1}^nc_ib'_i$, for
$(a_1\dots a_n)T=(c_1\dots c_n)$.

If $a_i=0$ the inequality (\ref{od}) is obvious. Suppose $a_i\ne0$. By Cramer's rule, we have $a_i=\det T' /\det T $, where $T'$ is the transition matrix from the basis $b_1, \dots,b_{i-1},x,b_{i+1},\dots,b_n$ to $\bb'$. By Hadamard's inequality, we get
$$
|\det T'|\le \sum_{j\ne i}\v{b_j}+\v{x}-\vol(E).
$$
Hence,
\begin{align*}
\v{a_ib_i}=&\,|a_i|+\v{b_i}=|\det T'|-|\det T|+\v{b_i}\\\le &\,\v{x}+\vol(\bb)-\vol(E)-|\det T|=\v{x}+\od(\bb).
\end{align*}
\end{proof}

\begin{theorem}\label{reducecriteria}
A basis $\bb$ is reduced if and only if $\od(\bb)=0$.

In this case, $|d(\bb)|=\sum_{b\in\bb}\lceil \v{b}\rceil$.
\end{theorem}

\begin{proof}
If $\od(\bb)=0$, the lemma above shows that $\bb$ is reduced.

Suppose the basis $\bb$ is reduced. Let $m_i=-\lceil\v{b_i}\rceil\in\Z$, so that the basis $\bb'=(t^{m_1}b_1,\dots,t^{m_n}b_n)$ is orthonormal. If we take $m=\sum_{i=1}^nm_i$ then, clearly
$$
\vol(\bb')=m+\vol(\bb),\quad 0=|d(\bb')|=m+|d(\bb)|.
$$
Therefore, $\od(\bb)=\od(\bb')=0$, and $|d(\bb)|=-m$.
\end{proof}

\section{Reduction algorithm}\label{reductionalgorithm}
A reduction algorithm transforms any family of nonzero vectors in a normed space into a reduced one, still generating the same $A$-module.

In the literature there are several reduction algorithms for particular normed spaces \cite{MS,Len,Schoe,Gath}. 
In this section, our goal is to describe such a reduction algorithm for arbitrary real-valued normed spaces.

For the reader's commodity we assume that the initial family of nonzero vectors is a basis of the normed space. The reduction algorithm is based on an iterated performance of a reduction step.

\begin{definition}[Reduction step]\label{redstep}
Let $\bb$ be a basis of a normed space $(E,\v{\ })$. A reduction step is a replacement of some $b\in \bb$ by $\tilde{b}=b+\alpha$, for some $A$-linear combination $\alpha$ of $\bb\setminus\{b\}$ such that $\v{\tilde{b}}<\v{b}$.
\end{definition}
Clearly, $\left(\bb\setminus\{b\}\right)\cup\{\tilde{b}\}$ is still a basis of the lattice $L=\langle\bb\rangle_A$. Any reduction step keeps invariant the value $|d(\bb)|$ and decreases the value $\vol(\bb)=\sum_{b\in\bb}\v{b}$ strictly. Since $O\!D(\bb)=\vol(\bb)-\vol(E)-|d(\bb)|$ is bounded by $0$ from below, after a finite number of reduction steps we obtain a reduced basis of $L$ by Theorem \ref{reducecriteria} and Corollary \ref{discreteintervall}.

In practice, we work out this problem by using coordinates with respect to an orthonormal basis of $E$. We have then an explicit isometry between $E$ and the normed space $\perp_{i=1}^n\mathcal{K}(r_i)$, where $-1<r_1\le r_2\le\cdots \le r_n\le0$ are the lengths of the given orthonormal basis of $E$. Hence, we may assume that $E=\perp_{i=1}^n\mathcal{K}(r_i)$.

The initial basis $\bb$ is given by the rows of some $T\in\gl_n(K)$, and the reduction algorithm finds $R\in\gl_n(A)$ such that the rows of $RT$ are a reduced basis $\widetilde{\bb}$. The matrix $R=T(\widetilde{\bb}\rightarrow \bb)$ is obtained as a product, $R=R_m\cdot R_{m-1}\cdots R_1$, where each $R_i$ represents the concatenation of several reduction steps. 

\subsection{The case $\#\sig(E)=1$}\label{easycase}
Let $E=\kk^n(r)$ for some $-1<r\le0$, with norm: 
$$
\v{(a_1,\dots, a_n)}=\max_{1\leq i\leq n}\{|a_i|\}+r.
$$
Since a basis of $E$ is reduced if and only if it is reduced as a basis of $\kk^n$, we could assume that $r=0$. Although there exist several descriptions of a reduction algorithm for this particular normed space \cite{Len,MS}, we review it in the case $r\neq 0$, in regard to its generalization to arbitrary normed spaces. 

The standard basis $(e_1,\dots,e_n)$ of $K^n$ is an orthonormal basis of $E$. A vector $(a_1,\dots,a_n)=\sum_{i=1}^na_ie_i$ belongs to $E_{\le r}$ if and only if $|a_i|\le0$ for all $i$. Hence, Corollary \ref{basisV} shows that $(\red_r(e_1),\dots,\red_r(e_n))$ is a $k$-basis of $V_r=E_{\leq r}/E_{<r}$, and the choice of this basis determines a $k$-linear isomorphism:
$$
V_r\longrightarrow k^n,\quad \red_r(a_1,\dots,a_n)\ \mapsto\ \big(\zc(a_1),\dots,  \zc(a_n)\big).
$$
Therefore, Theorem \ref{SS} provides a comfortable criterion to decide whether a basis of $E$ is reduced or not.
\begin{corollary}\label{nicecri}
A basis $(b_1,\dots,b_n)$ of $ E$ is reduced if and only if the matrix
$$
\left(\zc\left(t^{-\lceil\v{b_i}\rceil}b_{i,j}\right)\right)_{1\leq i,j\leq n}\in k^{n\times n}
$$
has rank $n$, where $b_i=(b_{i1},\dots, b_{in})$ for $1\leq i\leq n$.
\end{corollary}

\begin{example}\label{exone}
\rm{Let $K=\Q(t)$ and $E=\kk^2$. We consider $\bb=(b_1,b_2)$ with
$$
b_1=\left( 2t+1  , 1 \right),\quad b_2=( t^7+2, 2t^6).
$$
Clearly, $\v{b_1}=1$ and $\v{b_2}=7$. We consider
$$
M=\left(\begin{array}{cc} \zc\big(\frac{2t+1}{t}\big) &  \zc\big(\frac{1}{t}\big)\\ \zc\big(\frac{t^7+2}{t^7}\big)&   \zc\big(\frac{2}{t}\big)\end{array}\right)  =\left(\begin{array}{cc}2 & 0 \\1 & 0\end{array}\right)\in \Q^{2\times 2}.
$$
Since $ \mathrm{rank}(M)<2$, Corollary \ref{nicecri} shows that the basis $\bb$ is not reduced. }
\end{example}

Let us describe a concrete procedure to perform the reduction steps.

We order by increasing length the vectors $b_1,\dots, b_n$ of the input basis $\bb$. For $1\leq i\leq n$, let $b_i=(b_{i1},\dots, b_{in})$. We transform the  matrix
$$
M=\left(\zc(t^{-\lceil\v{b_i}\rceil}b_{i,j})\right)_{1\leq i,j\leq n}\in k^{n\times n},
$$
into row echelon form, $M'=PM$, with $P=(p_{i,j})$ belonging to the set $\mathrm{LT}_n(k)$ of lower triangular matrices with diagonal entries equal to $1$, up to a permutation of its rows. For commodity of the reader, we discuss only the case where $P$ is already a lower triangular matrix.

The rows of $P$ which correspond to the zero-rows of $M'$ give us non-trivial expressions of the zero vector in $k^n$ as $k$-linear combinations of the rows of $M$. This corresponds to non-trivial expressions of the zero vector in $V_r$ as $k$-linear combinations of $\red_r(t^{-\lceil\v{b_1}\rceil}b_1),\dots,$ $ \red_r(t^{-\lceil\v{b_n}\rceil}b_n)$.

Let $m=\mathrm{rank}(M)$. Let $P_1,\dots,P_n$ be the rows of $P$, and consider the lower triangular matrix $P'$ with rows $P'_1,\dots,P'_n$ defined by
$$
P'_j=\begin{cases}
e_{j},&\text{ if } j\leq m,\\
P_j=(p_{j,1}\cdots p_{j,j-1}\ p_{j,j}=1\ 0\cdots 0), &\text{ if } j> m.
\end{cases}
$$
For $1\leq j\leq m$ we take $\tilde{b}_{j}=b_{j}$ while for $m<j\leq n$ we consider
\begin{align}\label{refreducstep}
\tilde{b}_{j}=\sum_{i<j}p_{j,i}t^{\lceil\v{b_{j}}\rceil-\lceil \v{b_i} \rceil}b_i+b_{j}.
\end{align}
The family $\tilde{\bb}=(\tilde{b}_{1},\dots,\tilde{b}_{n})$ is a basis of the lattice $\gen{\bb}_A$, and the transition matrix $R=T(\widetilde{\bb}\rightarrow \bb)$ is given by
\begin{align}\label{DefR}
R&=\dia(t^{\lceil\v{b_{1}}\rceil},\dots,t^{\lceil\v{b_{n}}\rceil})\cdot P'\cdot \dia(t^{-\lceil\v{b_1}\rceil},\dots,t^{-\lceil\v{b_n}\rceil}).
\end{align}
Note that $R$ is a lower triangular matrix with diagonal entries equal to $1$ and it belongs to $ \gl_n(A)$ thanks to our assumption $\v{b_1}\leq\cdots\leq\v{b_n}$.

By construction, $\red_r(t^{-\lceil\v{b_j}\rceil}\tilde{b}_j)=0$, so that $\v{\tilde{b}_{j}}<\v{b_{j}}$ and (\ref{refreducstep}) is a reduction step. Thus, this procedure performs $n-\mathrm{rank}(M)$ reduction steps at once.

\begin{example}
\rm{We consider Example \ref{exone} again. For the matrices
$$
P=\left(\begin{array}{cc}1 & 0 \\-\frac{1}{2} & 1\end{array}\right)\text{ and }M'=\left(\begin{array}{cc}2 & 0 \\0 & 0\end{array}\right)
$$
it holds $PM=M'$ and $M'$ is in row echelon form. Then, the matrix 
\begin{align*}
R=\dia(t,t^7)\cdot P\cdot \dia(t^{-1},t^{-7})=\left(\begin{array}{cc}1 & 0 \\-\frac{t^6}{2} & 1\end{array}\right)\in\gl_2(\Q[t])
\end{align*}
realizes a reduction step $(\tilde{b}_1\ \tilde{b}_2)^\tp=R\cdot (b_1\ b_2)^\tp$. We get
$$
\tilde{b}_1=b_1,\quad  \tilde{b}_2= \frac{-t^6}{2} b_1+b_2=\left(-\frac{t^6}{2}+2 ,\frac{3t^6}{2}\right).
$$
Since $\v{\tilde{b}_2}=6$, we obtain 
$$
\left(\zc(t^{-\lceil \v{\tilde{b}_i}\rceil }\tilde{b}_{i,j})\right)_{1\leq i,j\leq 2}=\left(\begin{array}{cc} \zc\big(\frac{2t+1}{t}\big) &  \zc\big(\frac{1}{t}\big)\\ \zc\big(-\frac{1}{2}+\frac{2}{t^6}\big)&   \zc(\frac{3}{2})\end{array}\right)  =\left(\begin{array}{cc}2 & 0 \\-\frac{1}{2} & \frac{3}{2}\end{array}\right).
$$
Since this matrix has rank $2$, the basis $(\tilde{b}_1,\tilde{b}_2)$ is reduced by Corollary \ref{nicecri}. }
\end{example}

\subsubsection*{The algorithm}
The initial basis is given by the rows $T_1,\dots,T_n$ of a matrix $T\in \gl_n(K)$. We may always assume that $T$ has polynomial entries. In fact, for $1\le i\le n$, let $g_i\in A$ be the least common multiple of the denominators of the entries in the $i$-th column of $T$, and denote $s_i=r_i-|g_i|$. The isometry
$$
\perp_{i=1}^n\mathcal{K}(r_i)\,\longrightarrow\, \perp_{i=1}^n\mathcal{K}(s_i),\quad(a_1,\dots,a_n)\mapsto (a_1g_1,\dots,a_ng_n) 
$$  
sends the lattice generated by the rows of $T$ to the lattice generated by the rows of $T\op{diag}(g_1,\dots,g_n)$, which has polynomial entries.

\begin{algorithm}\caption{: Basis reduction for $E=\kk^n(r)$}\label{Algo0}\begin{algorithmic}[1]

\REQUIRE $T\in \gl_n(K)\cap A^{n\times n}$.\ENSURE Reduced basis of the lattice generated by the rows of $T$.\\[0.25 cm]\STATE  $s\leftarrow 1$\WHILE{$s< n$}\label{for}\STATE Sort rows of $T$ increasingly ordered w.r.t. $\v{~}$\STATE $M\leftarrow (\zc(t^{-\lceil\v{T_i}\rceil}t_{i,j}))_{1\leq i,j\leq n}\in k^{n\times n}$\STATE Compute $P=(p_{i,j})\in \mathrm{LT}_n(k)$ s.t. $M':=PM$ is in row echelon form\STATE $s\leftarrow \mathrm{rank}(M')$ \IF{$s<n$}\FOR{$i=s+1,\dots,n$}\label{loop}\STATE  $u_i\leftarrow \max\{1\leq j\leq n\mid p_{i,j}\neq 0\}$\STATE $T_{{u_i}}\leftarrow T_{{u_i}}+\sum_{j=1}^{u_i-1}t^{\lceil \v{T_{u_i}}\rceil-\lceil\v{T_{j}}\rceil}\cdot  p_{i,j} T_{j}$\ENDFOR\ENDIF\ENDWHILE\STATE \textbf{return} $T$

\end{algorithmic}\end{algorithm}

\subsection{The general case}

Let $E=\bot_{l=1}^\kappa \kk^{m_l}(r_l)$ for some $-1<r_1<\cdots<r_\kappa\leq 0$. 
For all $1\le l\le \kappa$, denote $n_k=m_1+\cdots+m_l$, and let $n=n_\kappa=\dim E$.

The standard basis $(e_1,\dots,e_n)$ of $K^n$ is an orthonormal basis of $E$. By Corollary \ref{basisV}, the vectors $\left(\red_{r_l}(e_j)\right)_{n_{l-1}<j\le n_l}$ are a basis of $V_{r_l}$ for each $1\le l\le \kappa$. The choice of this basis yields a $k$-linear isomorphism: 
$$
V_{r_l}\longrightarrow k^{m_l},\quad
\red_{r_l}(t^{-\lceil\v{b}\rceil}b) \ \mapsto\ (\zc(t^{-\lceil\v{b}\rceil}a_j))_{n_{l-1}<j\le n_l},
$$
where $b=(a_1,\dots,a_n)\in K^n$ has length $\v{b}\equiv r_l\md{\Z}$.

Therefore, we can reinterpret Theorem \ref{SS} as follows:
\begin{corollary}\label{gencri}
Let $\bb=(b_1,\dots,b_n)$ be a basis of $E$ ordered by increasing length, with $b_i=(b_{i,1}, \dots, b_{i,n})\in K^n$ for all $i$.
The basis $\bb$ is reduced if and only if for all $1\le l\le\kappa$ the following matrix  
has rank $m_{l}$:
$$
M_{r_l}:=\left(\zc(t^{-\lceil\v{b_i}\rceil}b_{i,j})\right)_{i\in I_\bb(r_l),n_{l-1}< j\le n_{l}},
$$
where $I_\bb(r_l)=\{1\leq i\leq n\mid \v{b_i}\equiv r_l \bmod\Z\}$.
\end{corollary}

\subsubsection*{The algorithm}
The initial basis $\bb$ is given by the rows $T_1,\dots,T_n$ of a matrix $T\in \gl_n(K)$. As argued for Algorithm \ref{Algo0}, we may always assume that $T$ has polynomial entries.

We split the basis $\bb$ of $E$ into subsets $\bb_r=\{b\in\bb\mid \v{b}\equiv r \bmod\Z\}$ for any $r\in \{r_1,\dots,r_\kappa\}$, and apply for each of these subsets reduction steps as we did in  Algorithm \ref{Algo0}. Unfortunately, the length of a reduced vector $b+\alpha$ may not lie in the same class as $\v{b}$ modulo $\Z$. Therefore, it may happen that the subsets $\bb_r$ change after any reduction step. 

Recall that $\mathrm{LT}_n(k)$ is the set of all $P\in \gl_n(k)$ which are lower triangular with $1$ at the diagonal, up to row permutation.

\begin{algorithm} \caption{Basis reduction for $E=\bot_{l=1}^\kappa \kk^{m_l}(r_l)$}\label{Algo1}\begin{algorithmic}[1]

\REQUIRE $T\in\gl_n(K)\cap A^{n\times n}$.		
\ENSURE Reduced basis of the lattice generated by the rows of $T$.\\[0.25 cm]
\STATE   $\mathrm{vals}\leftarrow [r_1,\dots,r_\kappa]$
\STATE  $\iota\leftarrow 1$
\WHILE{$\iota\leq \#\mathrm{vals}$}\label{for}
\STATE $\bb\mathrm{vals}\leftarrow[  \v{T_1},\dots,\v{T_n}]$
\STATE Sort $\bb\mathrm{vals}$ increasingly ordered and apply changes to the rows of $T$
\STATE Determine $1\leq l\leq \kappa$ with $\mathrm{vals}[\iota]\equiv r_l\bmod \Z$
\STATE Determine all $1\leq e_1,\dots,e_f\leq n$ with $\bb \mathrm{vals}[e_i]\equiv \mathrm{vals}[l]\bmod \Z$
\STATE $M\leftarrow (\zc(t^{-\lceil\bb\mathrm{vals}[e_i]\rceil }t_{e_i,j}))_{1\leq i\leq f,n_{l-1}< j\leq n_l}\in k^{f\times m_l}$ 
\STATE Compute $P=(p_{i,j})\in \mathrm{LT}_f(k)$ s.t. $M':=PM$ is in row echelon form
\STATE $s\leftarrow \mathrm{rank}(M')$\IF{$s=f$}\label{remm1}\IF{$f<m_l$ and $\mathrm{vals}[\iota] \notin \{\mathrm{vals}[s] \mid s> \iota\}$}\label{remm2}
\STATE $\mathrm{Append}(\mathrm{vals},\mathrm{vals}[\iota] )$
\ENDIF\ELSE
\FOR{$i=s+1,\dots,f$}\label{loop}
\STATE  $u_i\leftarrow \max\{1\leq j\leq f\mid p_{i,j}\neq 0\}$\label{strange}
\STATE $T_{e_{u_i}}\leftarrow T_{e_{u_i}}+\sum_{j=1}^{u_i-1}t^{\lceil\bb\mathrm{vals}[e_{u_i}] \rceil-\lceil \bb\mathrm{vals}[e_j]\rceil} p_{i,j} T_{e_j}$\label{line21}
\STATE $\bb\mathrm{vals}[e_{u_i}]\leftarrow \v{T_{e_{u_i}}}$\IF{$\bb\mathrm{vals}[e_{u_i}]-\lceil\bb\mathrm{vals}[e_{u_i}]\rceil\notin \{\mathrm{vals}[s]\mid s> \iota\}$}\label{remm3}
\STATE $\mathrm{Append}(\mathrm{vals},\bb\mathrm{vals}[e_{u_i}] -\lceil\bb\mathrm{vals}[e_{u_i}]\rceil)$
\ENDIF\ENDFOR\ENDIF
\STATE $\iota\leftarrow \iota+1$
\ENDWHILE\STATE \textbf{return} $T$

\end{algorithmic}\end{algorithm}
\newpage
Let us add some comments to clarify some parts of the algorithm.

\noindent{\bf Steps 12-14. }If no reduction step can be a applied but the number of vectors in $\bb$ of length $r_l \bmod \Z$ is lower than $m_l$, we have not found enough vectors in the set $\bb_{r_l}$. Later, there will occur (after several reduction steps) new vectors with length $r_l$ modulo $\Z$. Therefore, we must reconsider the value $r_l=\mathrm{vals}[\iota]$ afterwards.

\noindent{\bf Steps 20-22. }If the length $r$ of the reduced vector does not coincide with the length of the original vector modulo $\Z$. Then, we have to reconsider the class $r$ mod $\Z$ later.

\begin{remark}
By Proposition \ref{basic}, for a reduced basis $\bb=(b_1,\dots,b_n)$ the values $\v{b_i}$, $1\leq i\leq n$, are the successive minima of $L$. Moreover, for a real number $r$, Proposition \ref{basic} shows that the $k$-vector space $L_{\leq r}$ admits the basis
$$
\{b_it^{j_i}\mid 1\leq i\leq n,\quad 0\leq j_i\leq \lfloor r-\v{b_i}\rfloor\}.
$$
 Hence, Algorithm \ref{Algo1} can also be adapted to compute these objects.
\end{remark}

Let us illustrate the algorithm with an example.

\begin{example}\label{ex3}

\rm{Let $K=\F_3(t)$ be the rational function field over $\F_3$, the finite field of three elements. We consider the normed space $$E=\kk(-1/2)\,\bot\,\,\kk(-1/3)\,\bot\,\, \kk(-1/4).$$ We have $r_1=-1/2$, $r_2=-1/3$, and $r_3=-1/4$ with multiplicities $m_l=1$ for $1\leq l\leq 3$. Consider the following basis $\bb=(b_1,b_2,b_3)$ of $E$:
$$
b_1=\left(t^2 ,t^2+1,0\right),\quad b_2=\left(t(t^2+1),t,t^4+1\right),\quad b_3=\left(0,t^4(t+1),t^4\right).
$$
The norm on $E$ is given by $\v{(a_1,a_2,a_3)}=\max\{|a_1|-1/2,|a_2|-1/3,|a_3|-1/4\}$; hence, $\v{b_1}=5/3$, $\v{b_2}=15/4$ and $\v{b_3}=14/3$.  

The basis $\bb$ is not reduced, as $\bb$ contains no vector of length in $r_1+\Z$. We apply a reduction step focussing our attention on the set $\bb_{r_2}=\{b_1,b_3\}$. We consider 
$$
M_{r_2}=\left(\begin{array}{c}\zc\big(t^{-\lceil 5/3\rceil}(t^2+1)\big )  \\\zc\big(t^{-\lceil 14/3\rceil}t^4(t+1)\big)\end{array}\right)=\left(\begin{array}{c}1 \\1 \end{array}\right)\in \F_3^{2\times 1 }
$$
and transform $M_{r_2}$ into row echelon form, $P M_{r_2}=M'$, with
$$
P=\left(\begin{array}{cc}1  & 0  \\2 &  1\end{array}\right),\quad M'=\left(\begin{array}{c}1  \\0 \end{array}\right).
$$
We perform the reduction step $(\tilde{b}_1\  \tilde{b}_3)^\tp=R\cdot (b_1 \ b_3)^\tp$ with the transition matrix $R$  defined as in (\ref{DefR}):
$$
R=\dia(t^2,t^5)\cdot P\cdot \dia(t^{-2},t^{-5})=\left(\begin{array}{cc}1 & 0   \\2t^3  & 1\end{array}\right)\in\gl_2(\F_3[t]).
$$
We obtain $\tilde{b}_1=b_1$ and 
$\widetilde{b}_3=2t^3\cdot b_1+b_3=\left(2t^5 ,t^3(t+2) ,t^4\right)$, with $\v{\tilde{b}_3}=7/2$. Note that $\tilde{b}_3$ and $b_3$ do not have the same length modulo $\Z$.

The basis $(\tilde{b}_1,b_2,\tilde{b}_3)$ is reduced, since $\v{\tilde{b}_1}, \v{b_2}$, $\v{\tilde{b}_3}$ are different modulo $\Z$. }
\end{example}

\subsection{Complexity}
We are interested in the complexity of Algorithms \ref{Algo0} and \ref{Algo1}. All estimations are expressed in the number of necessary operations in $k$. Recall that $\sig(E)$ denotes the set of different lengths modulo $\Z$ of all nonzero vectors in the normed space $(E,\vv)$.
\begin{lemma}\label{loopnumber} 
Let $\bb$ be a basis of an $n$-dimensional normed space $E$. The number of reduction steps to transform $\bb$ into a reduced basis is bounded by 
$$
\#\sig(E)\cdot \lfloor O\!D(\bb)\rfloor +(\#\sig(E)-1)n.
$$
\end{lemma}
\begin{proof}
Let $\bb=(b_1,\dots,b_n)$ and let $\widetilde{\bb}=(\widetilde{b}_1,\dots,\widetilde{b}_n)$ be a reduced basis obtained from $\bb$. Each vector $b_i$ is changed by several reduction steps until we obtain the vector $\widetilde{b}_i\in \widetilde{\bb}$. Let us denote by $R_i$ the number of these reduction steps; that is
$$
b_i\rightarrow b_i^{(1)}\rightarrow\dots \rightarrow b_i^{(R_i)}=\widetilde{b}_i.
$$
If we denote $D_i:=\v{b_i}-\v{\widetilde{b}_i}$, then $O\!D(\bb)=D_1+\cdots +D_n$.

Let $\kappa:=\#\sig(E)$. If we apply $\kappa$ consecutive reduction steps to any vector $b\in \langle \bb\rangle_A$:
\begin{align}\label{veclengtho}
b=b^{(0)}\rightarrow b^{(1)}\rightarrow \cdots \rightarrow b^{(\kappa)}
\end{align}
then, $\v{b}-\v{b^{(\kappa)}}\geq 1$. In fact, since the lengths of all nonzero vectors in $E$ have only $\kappa$ possibilities modulo $\Z$, among the $\kappa+1$ vectors in (\ref{veclengtho}) there must be a coincidence. If $0\leq j< l\leq \kappa$ satisfy $\v{b^{(l)}}\equiv \v{b^{(j)}} \bmod \Z$ then:
$$
 \v{b}-\v{b^{(\kappa)}}\geq  \v{b^{(j)}}-\v{b^{(l)}}\geq 1.
$$
This argument shows that $R_i\leq \lfloor D_i \rfloor\kappa+\kappa-1$. Therefore, the total number of reduction steps is $R_1+\cdots+R_n\leq  \lfloor O\!D(\bb) \rfloor\kappa+(\kappa-1)n$.
\end{proof}

We introduce heights of rational functions in order to measure the complexity of the reduction algorithms.
\begin{definition}\label{matrixhight}
For $g=f/h\in K$, with coprime polynomials $f,h\in A$, we define the \emph{height} of $g$ by
$$
h(g):=\max\{|f|,|h|\}.
$$
The \emph{height} of a matrix $T=(t_{i,j})\in K^{n\times m}$ is defined to be
$$
h(T):=\max\{h(t_{i,j})\mid 1\leq i\leq n,\quad 1\leq j\leq m\}.
$$
\end{definition}
The next lemma presents some properties of the height, which will be useful for the complexity analyses of subsequent algorithms.

\begin{lemma}\label{hightinversemat}
Let $T,T'\in K^{n\times n}$. 
\begin{enumerate}
\item $h(T\cdot T')\leq h(T)+h(T').$
\item If $T$ is invertible, then
$ |\det T|$, $|\det T^{-1}|$, $h(T^{-1})\leq nh(T)$.
\end{enumerate}
\end{lemma}

\begin{proof}
The first statement is obvious. Suppose that $T$ is invertible. For any permutation $\sigma$ of $\{1,2,\dots,n\}$ we have 
$$
\pm |t_{1,\sigma(1)}\cdots t_{n,\sigma(n)}|=\pm\sum_{i=1}^n|t_{i,\sigma(i)}|\leq \sum_{i=1}^nh(t_{i,\sigma(i)})\leq nh(T).
$$
This shows that $\pm |\det(T) |\leq nh(T)$; thus, $|\det(T^{-1})|=-|\det(T) |\leq nh(T)$.

Denote by $T_{i,j}$ the matrix which arises from deleting the $i$-th row and the $j$-th column in $T$. The entries $s_{i,j}$ of $T^{-1}$ may be computed as
$$
s_{i,j}=(-1)^{i+j}\frac{\det (T_{j,i})}{\det (T)}.
$$
Hence, $h(s_{i,j})=\max\{|\det (T_{j,i})|,|\det (T)|\}\leq nh(T)$.
\end{proof}

\begin{lemma}\label{odbbound1}
Let $\bb$ and $\bb'$ be bases of the $n$-dimensional normed space $E$ and let $\bb'$ be orthonormal. Denote by $T$ the transition matrix from $\bb$ to $\bb'$. Then, $O\!D(\bb)< n(2h(T)+1)$.
\end{lemma}

\begin{proof}
Let $\bb=(b_1,\dots,b_n)$ and $\bb'=(b'_1,\dots,b'_n)$. By definition,
\begin{align}\label{odinlem}
O\!D(\bb)=\sum_{i=1}^n\v{b_i}-\vol(E)-|\det (T)|.
\end{align}
With $T=(t_{i,j})$ we obtain, for $1\leq i\leq n$:
$$
\v{b_i}=\max_{1\leq j\leq n}\{|t_{i,j}|+\v{b'_j}\}\leq \max_{1\leq j\leq n}\{|t_{i,j}|\}\leq h(T).
$$
Hence, $\sum_{i=1}^n\v{b_i}\leq nh(T)$. 
On the other hand, $\vol(E)=\sum_{i=1}^n\v{b_i'}>-n$, since $-1<\v{b'_i}\leq 0$ for all $i$, as $\bb'$ is orthonormal.
Finally, $-|\det(T)|\leq nh(T)$ by item (2) from Lemma \ref{hightinversemat}. Therefore, from (\ref{odinlem}) we deduce $O\!D(\bb)<nh(T)+n+nh(T)= n(2h(T)+1)$.
\end{proof}

\begin{lemma}\label{compl}
Let $\bb'$ be an orthonormal basis of an $n$-dimensional normed space $(E,\v{~})$ and let $\bb$ be a basis of $E$ such that the transition matrix $T=T(\bb\to\bb')$ has polynomial entries. Then, Algorithm \ref{Algo1} takes at most
$$
O(\#\sig(E)( n^4 \cdot h(T)+ n^3 \cdot h(T)^2))
$$
arithmetic operations in $k$ to transform $\bb$ into a reduced basis.
\end{lemma}

\begin{proof}
By any reduction step in Algorithm \ref{Algo1} the value $O\!D(\bb)$ is decreased strictly. If $\kappa=\#\sig(E)$, according to Lemma \ref{loopnumber} and Theorem \ref{reducecriteria}, the set $\bb$ is reduced after at most $ \lfloor O\!D(\bb)\rfloor\kappa+(\kappa-1)n$ steps. 

Clearly, the runtime of the algorithm is dominated by the transformation of matrices into row echelon form and the realization of reduction steps.

At first we analyze the complexity of the transformation of matrices into row echelon form. Denote by $r_1,\dots,r_\kappa$ the different lengths of vectors in $\bb'$ and $m_1,\dots,m_\kappa$ its multiplicities, so that $n=m_1+\cdots +m_\kappa$. Suppose, that after $i-1$ steps in Algorithm \ref{Algo1} we have transformed the basis $\bb$ into $\bb_i=(b_{i_1},\dots,b_{i_n})$. We can split $\bb_i$ into disjoint subsets
$$
\bb_i=\bb_{r_1}\cup\cdots\cup\bb_{r_\kappa},
$$
where $\bb_{r_k}:=\{b\in\bb_i\mid \v{b}\equiv r_k\bmod \Z\}$, for $1\leq k\leq \kappa$. Assume $\bb_i$ is not reduced. By Corollary \ref{gencri}, for at least one $r_k$, the matrix $M_{r_k}\in k^{\# I_{\bb_i}(r_k)\times m_k}$ has not full rank.
 
In the worst case, we have to transform all matrices $M_{r_1},\dots,M_{r_\kappa}$ into row echelon form until we detect at least one reduction step (i.e. one zero row). The cost for transforming all $M_{r_j}$, $1\leq j\leq \kappa$, into row echelon form is less than or equal to the cost of transforming one $n\times n$ matrix over $k$ into row echelon form (which is equal to $O(n^3)$ operations in $k$ \cite{H.C.}). Hence, the cost of all transformations of matrices into row echelon form along Algorithm \ref{Algo1} is bounded by $O((O\!D(\bb)\kappa+(\kappa-1)n)\cdot n^3)$ operations in $k$. According to Lemma \ref{odbbound1} the last complexity bound can be estimated by $O(\kappa n^4 h(T))$.

Additionally, we compute $A$-linear combinations of the rows of $T$ (line $18$ of Algorithm \ref{Algo1}), where the coefficients are of the form $\alpha t^m$ with $\alpha \in k$ and a nonnegative integer $m$. After any reduction step the degree of the entries in $T$ is less or equal than before; that is, at any level the value of $h(T)$ is not increased. Since the multiplication of a polynomial by a $t$-power is just a shift of the exponents, we can consider the latter $A$-linear combinations of rows of $T$ as $k$-linear combinations. 

The cost of any reduction step applied to the rows of $T$ is $O(n^2h(T))$ operations in $k$. Thus, the total cost of performing all reduction steps of Algorithm \ref{Algo1} is $O(\kappa n^3h(T)^2)$. this ends the proof of the lemma.
\end{proof}

\begin{remark}\label{denominators}
If the transition matrix $T(\bb\to\bb')$ does not belong to $A^{n\times n}$, we must add the cost of finding the $\lcm$ of the entries of each column and the cost of multiplying by them to get rid of denominators. The total cost of the reduction algorithm is then
$O\big(\#\sig(E)\cdot n^4 \cdot  h(T)^2\big)$ operations in $k$.
\end{remark}

In Subsection \ref{semisection} we will present an optimized version of the reduction algorithm (cf. Lemma \ref{reductionspeziell}). 

If $\#\sig(E)=1$, Algorithm \ref{Algo1} coincides with Algorithm \ref{Algo0}. Hence, the complexity bounds for the latter follow immediately from Lemma \ref{compl}.

\begin{corollary}\label{copleximegaeasy}
For $r\in\R$, let $\bb=(b_1,\dots,b_n)$ be a basis of the normed space $E=\kk^n(r)$ such that $T=(b_1 \dots b_n)^\tp$ belongs to $A^{n\times n}$. Algorithm \ref{Algo0} takes $O(n^4  h(T)+ n^3  h(T)^2)$ arithmetic operations in $k$ to transform $\bb$ into a reduced basis. \end{corollary}

In practice, the runtime of Algorithm \ref{Algo0} (and Algorithm \ref{Algo1}) is dominated by the realization of the reduction steps. The reason for this is that $h(T)\geq n$ in most of the cases. Under this assumption, the complexity of Algorithm \ref{Algo0} is equal to $O(n^3 h(T)^2)$ operations in $k$. In this context, our reduction algorithm is one magnitude better than the reduction algorithms described in \cite{Len,Gath} and its complexity coincides with the one in \cite{MS}. 

\subsection{Classes of lattices and semi-reduceness}\label{secapp}

In the sequel denote by $E$ an $n$-dimensional $K$-vector space. We consider a norm $\v{~}$ on $E$ and a lattice $L$ in $(E,\v{~})$. Our aim is to construct a \emph{semi-reduced basis} (cf. Definition \ref{semredd}) $\bb$ of $L$, which ``nearly" behaves as a reduced one. 

To this end, we shall consider instead an integer-valued lattice $(L,\vv') $, which almost coincides with $(L,\v{~})$. For instance, for the computation of the vector spaces $(L,\v{~})_{\leq r}$, for $r\in \Z$, it is sufficient to determine a reduced basis $\bb$ of the lattice $(L,\vv')$. Moreover, a reduced basis $\bb$ of $(L,\vv')$ can be used as a precomputation for the reduction algorithm in order to determine a reduced basis of $(L,\v{~})$. In this way, the reduction algorithm can be accelerated.

\begin{definition}\label{latticespacedef}
We define the \emph{norm space} $\NM(E)$ of $E$ as the set of all norms $\vv$ on $E$ such that $(E,\vv)$ becomes a normed space. The \emph{space of lattices} of $E$ is defined to be
$$
\LS(E):=\{(L,\v{~}) \text{ a lattice in }(E,\v{~})\mid \v{~}\in \NM(E) \}.
$$
\end{definition}\label{defequiv}
We introduce an equivalence class on $\NM(E)$.
\begin{definition}\label{equrel}
We say that two norms $\v{~}$ and $\v{~}'$ in $\NM(E)$ are \emph{equivalent}, and we write $\v{~}\sim \v{~}'$, if $\lceil\v{z}\rceil=\lceil\v{z}'\rceil$, for all $z\in E$.

In this case,  we write $(E,\v{~})\sim (E,\v{~}')$ and  $(L,\v{~})\sim (L,\v{~}')$. We say too that these two normed spaces or lattices are equivalent.
\end{definition}

The following results follow easily from the definitions.

\begin{lemma}\label{optimalbasis}
\begin{enumerate}
\item The relation $\sim$ is an equivalence relation on $\NM(E)$. 
\item If $(E,\v{~})$ is a normed space, then $(E,\lceil \v{~}\rceil)$ is a normed space. 
\end{enumerate}
\end{lemma}

Thus, in each equivalence class there is a unique integer-valued norm, defined by $z\mapsto \lceil\v{z}\rceil$ for any $\vv$ in the class. In particular, there are as many equivalence classes of norms as integer-valued norms

\begin{definition}
A basis $\bb$ of $E$ is called a \emph{semi-orthonormal} basis of $(E,\v{~})$, if it is, up to ordering, an orthonormal basis of a normed space $(E,\v{~}')$, which is equivalent to $(E,\v{~})$.
\end{definition}
Note that a semi-orthonormal basis of $(E,\v{~})$ is a semi-orthonormal basis of $(E,\v{~}')$, for all norms $\v{~}'$ in the class of $\v{~}$. In particular, an orthonormal basis is semi-orthonormal. 

\begin{lemma}\label{onbasissemiprop}
A basis $\bb$ of $(E,\v{~})$ is semi-orthonormal if and only if
\begin{align}\label{semiotisred}
\Big\lceil\V{\sum_{b\in\bb}a_bb}\Big\rceil=\max_{b\in\bb}\{|a_b|\}, \text{ for all }a_b\in K.
\end{align}
\end{lemma}

\begin{proof}
If $\bb$ is semi-orthonormal, there exists $\vv'\in\NM(E)$ with $\vv'\sim \vv$ such that $\bb$ is an orthonormal basis of $(E,\vv')$. Hence, 
$$
\V{\sum_{b\in\bb}a_bb}'=\max_{b\in\bb}\{\v{a_bb}'\}, \text{ for all }a_b\in K.
$$
As $-1<\v{b}'\leq 0$, for all $b\in\bb$, we obtain $\lceil \v{b}'\rceil=0$ and $\lceil \max_{b\in\bb}\{\v{a_bb}'\}\rceil=\max_{b\in\bb}\{|a_b|\}$. Since $\lceil\v{z}\rceil=\lceil\v{z}'\rceil$, for all $z\in E$, the statement holds.

Conversely, if $\vv$ satisfies (\ref{semiotisred}) then $\lceil \v{b}\rceil=0$ for all $b\in \bb$, and $\bb$ is an orthonormal basis of $(E,\vv')$, where $\vv'$ is the integer-valued norm defined by: $\v{z}'=\lceil\v{z}\rceil$.
\end{proof}

\begin{theorem}\label{sameval}
Let $\v{~}, \v{~}'\in \NM(E)$. It holds $\v{~}\sim \v{~}'$ if and only if the transition matrix from a semi-orthonormal basis of $(E,\v{~})$ to a semi-orthonormal basis of $(E,\v{~}')$ belongs to $\gl_n(A_\infty)$.
\end{theorem}

\begin{proof}
Denote by $\bb=(b_1,\dots,b_n)$ and by $\bb'=(b'_1,\dots,b'_n)$ semi-orthonormal bases of $(E,\v{~})$ and $(E,\v{~}')$, respectively. Let $T=(t_{i,j})$ be the transition matrix from $\bb$ to $\bb'$. For an arbitrary $z\in E$ we write $z=\sum_{i=1}^na_ib_i=\sum_{i=1}^na'_ib'_i$, with coefficients in $K$ such that $a'_i=\sum_{j=1}^nt_{j,i}a_j$. By (\ref{semiotisred}), the equality 
$$
\lceil\v{z}\rceil=\max_{1\leq i\leq n}\{|a_i|\}=\max_{1\leq i\leq n}\Big\{\Big|\sum_{j=1}^nt_{j,i}a_j\Big|\Big\}=\max_{1\leq i\leq n}\{|a'_i|\}=\lceil\v{z}'\rceil
$$
holds for all $z\in E$ if and only if $T\in\aut(\kk^n)$, and this group coincides with $\gl_n(A_\infty)$ by Theorem \ref{autgl}.
\end{proof}

\begin{lemma-definition}\label{smiredundred}
Let $\bb$ be a semi-orthonormal basis of $(E,\v{~})$. Then, we define $L_\infty:=\langle \bb\rangle_{A_\infty}=(E,\v{~})_{\leq 0}$. Moreover, any $A_\infty$-basis of $L_\infty$ is a semi-orthonormal basis of $(E,\v{~})$.
\end{lemma-definition}
\begin{proof}
By Lemma \ref{onbasissemiprop} it holds for $z=\sum_{i=1}^na_i b_i\in E$ with coefficients $a_i$ in $K$ that $\lceil \v{z}\rceil=\max_{1\leq i\leq n}\{|a_i|\}$. Clearly, $\v{z}\leq 0$ if and only if $|a_i|\leq 0$, for $1\leq i\leq n$; hence $L_\infty=(E,\vv)_{\leq 0}$. 

Since the transition matrix between two bases of $L_\infty$ belongs to $\gl_n(A_\infty)$, the second statement holds by Theorem \ref{sameval}. 
\end{proof}

\begin{definition}\label{semredd}
A subset $\{b_1,\dots,b_m\}$ in a normed space $(E,\v{~})$ is called \emph{semi-reduced} or \emph{weakly reduced} if 
$$
\Big\lceil\V{\sum_{i=1}^ma_ib_i}\Big\rceil=\max_{1\leq i\leq m}\{\lceil \v{a_ib_i}\rceil\},
$$
for any $a_1,\dots,a_m\in K$. Or equivalently, the subset is reduced with respect to the unique integer-valued norm equivalent to $\vv$.
\end{definition}

Clearly, any reduced set is semi-reduced. Many of the results concerning a reduced set can be adapted to semi-reduced sets. For instance, the next result follows immediately from the definitions.

\begin{lemma}\label{equivsemiot}\quad
\begin{enumerate}
\item A basis $\bb$ of a normed space $(E,\v{~})$ is semi-orthonormal if and only if $\bb$ is semi-reduced with $-1<\v{b}\leq 0$, for all $b\in \bb$. 
\item If $\bb=(b_1,\dots,b_n)$ is semi-reduced, then $(t^{-\lceil \v{b_1}\rceil}b_1,\dots,t^{-\lceil \v{b_n}\rceil}b_n)$ is semi-orthonormal.
\item If $\v{~}\sim\v{~}'$, then, any semi-reduced basis of $(L,\v{~})$ is a semi-reduced basis of $(L,\v{~}')$.
\end{enumerate}
\end{lemma}

The next theorem summarizes all data shared by all lattices in the equivalence class of  $(L,\vv)$.

\begin{theorem}\label{semisuccmin}
For $i\in \{1,2\}$, denote by $\bb_i=(b_{1,i},\dots,b_{n,i})$ a semi-reduced basis of the lattice $(L,\v{~}_i)$, which is ordered by increasing length. Then, the following statements are equivalent:
\begin{enumerate}
\item $\v{~}_1\sim \v{~}_2$,
\item $\lceil \v{b_{1,i}}_1\rceil=\lceil \v{b_{2,i}}_2\rceil$ for $1\leq i\leq n$,
\item $(L,\v{~}_1)_{\leq r}=(L,\v{~}_2)_{\leq r}$ for all $r\in \Z$, and\label{ie3semisucc}
\item $(E,\v{~}_1)_{\leq 0}=(E,\v{~}_2)_{\leq 0}$, with $E=\langle \bb_1\rangle_K=\langle \bb_2\rangle_K$.
\end{enumerate}
\end{theorem}
\begin{proof}

$(1)\Rightarrow (3)$. One can easily see that item \textit{3} of Proposition \ref{basic} is correct for a semi-reduced basis and an integer $r$. Thus, 
$$
(L,\v{~}_1)_{\leq r}=\langle\{b_{1,i}t^{j_i}\mid 1\leq i\leq n, 0\leq j_i\leq -\lceil \v{b_{1,i}}_1\rceil+r\}\rangle_k.
$$
By Lemma \ref{equivsemiot}, the set $\bb_1$ is also a semi-reduced basis of $(L,\v{~}_2)$. Hence,
$$
(L,\v{~}_2)_{\leq r}=\langle\{b_{1,i}t^{j_i}\mid 1\leq i\leq n, 0\leq j_i\leq -\lceil \v{b_{1,i}}_2\rceil+r\}\rangle_k.
$$
Since $\lceil \v{z}_1\rceil=\lceil \v{z}_2\rceil$ holds for all $z\in E$, we get $(L,\v{~}_1)_{\leq r}=(L,\v{~}_2)_{\leq r}$.

$(3)\Rightarrow (2)$. Let $r_1\leq \cdots\leq r_n$; $s_1\leq \cdots\leq s_n$, with $r_i=\lceil \v{b_{1,i}}\rceil$, $s_i=\lceil \v{b_{2,i}}\rceil$. Assume that $r_1=s_1,\dots,r_i=s_i$, but $r_{i+1}<s_{i+1}$. Then, Proposition \ref{basic} shows that $\dim_k (L,\v{~}_1)_{\leq  s_{i+1}-1}\neq \dim_k (L,\v{~}_2)_{\leq  s_{i+1}-1}$, wich contradicts (3).

$(2)\Rightarrow (1)$. This implication follows immediately from the definitions.

Finally let us show that $(1)\Leftrightarrow (4)$. By Theorem \ref{sameval}, (1) is equivalent to the fact that the transition matrices between semi-orthonormal bases of the two normed spaces belong to $\gl_n(A_\infty)$. By Lemma-Definition  \ref{smiredundred}, this condition is equivalent to $(E,\v{~}_1)_{\le0}=(E,\v{~}_2)_{\le0}$.
\end{proof}

As we have seen in the proof of the last theorem it is sufficient to compute a semi-reduced basis of $(L,\vv)$ in order to determine a basis of $(L,\vv)_{\leq r}$, for $r\in \Z$.

\subsection{Computation of (semi-) reduced bases}\label{semisection}

Let $\bb'$ be an orthonormal basis of $(E,\v{~})$ and $L$ be a lattice in $(E,\vv)$. In section \ref{reductionalgorithm} we already described an algorithm (cf. Algorithm \ref{Algo1}), which computes a reduced basis of $L$. According to Lemma \ref{compl} the runtime of the computation of a reduced basis of $L$ is minimal if $E\cong \kk(r)^n$, i.e. $\#\sig(E)=1$.

The computation of a semi-reduced basis amounts to the computation of a reduced basis of a normed space in this favourable situation. In fact, by Lemmas \ref{optimalbasis} and \ref{equivsemiot}, a reduced basis of $(E,\lceil\vv\rceil)$ is a semi-reduced basis of $(E,\v{~})$ and since $(E,\lceil\vv\rceil)$ is an integer-valued normed space, it is isometric to $\kk^n$.

We may use this idea to describe an optimized version of Algorithm \ref{Algo1}. Clearly, $\bb'$ is an orthonormal basis of $(E,\lceil \vv\rceil)$ too; hence, we may consider $\bb$ as a basis of $(E,\lceil \vv\rceil)$ and call Algorithm \ref{Algo0} for $T=T(\bb\to\bb')$. This results in a semi-reduced basis $\bb_{\mathrm{semi}}$ of $(L,\vv)$. We will see that transforming $\bb_{\mathrm{semi}}$ into a reduced basis $\bb_\mathrm{red}$ of $(L,\vv)$ by Algorithm \ref{Algo1} can be realized at minimal cost. We summarize the results by the following pseudocode:


\newpage

\begin{algorithm}\caption{: Basis reduction}\label{Algo2}\begin{algorithmic}[1]

\REQUIRE $\bb'$ orthonormal basis of a normed space $(E,\v{~})$ and $\bb$ a basis of $E$. \ENSURE Reduced basis of the lattice $L=\langle \bb \rangle_A$.\\[0.25 cm]\STATE  $T_{\mathrm{semi}}\leftarrow$Algorithm \ref{Algo0}$(T(\bb\to\bb'))$\STATE $T_{\mathrm{red}}\leftarrow$Algorithm \ref{Algo1}$(T_{\mathrm{semi}})$\STATE \textbf{return} $T_{\mathrm{red}}$, transition matrix from a reduced basis of $L$ to $\bb'$

\end{algorithmic}\end{algorithm}

\begin{lemma}\label{shortodb}
Let $\bb$ be a semi-reduced basis of a lattice $(L,\v{~})$. Then, the orthogonal defect of $\bb$ satisfies $\od(\bb)<n$.
\end{lemma}

\begin{proof}
Let $\bb=(b_1,\dots,b_n)$ and consider a reduced basis $\bb'=(b'_1,\dots,b'_n)$ of $L$. Assume that both bases are increasingly ordered with respect to the length of their vectors. By Theorem \ref{reducecriteria}, we obtain $\od(\bb')=0$, since $\bb'$ is reduced. 
Hence, $\vol(\bb')=\vol(E)+|d(L)|$.
According to Theorem \ref{semisuccmin} we obtain $\lceil\v{b_i}\rceil=\lceil\v{b'_i}\rceil$ and therefore $\v{b_i}<\v{b'_i}+1$, for $1\leq i\leq n$. Therefore, 
$$
\od(\bb)=\vol(\bb)-\vol(E)-|d(L)|=\vol(\bb)-\vol(\bb')<n.
$$
\end{proof}

According to Lemma \ref{loopnumber}, the last lemma shows that at most $O(\#\sig(E) n)$ reduction steps are necessary to transform a semi-reduced basis of $(L,\vv)$ into a reduced one. Having in mind that $\#\sig(E)\leq n$, Corollary \ref{copleximegaeasy} and the proof of Lemma \ref{compl} yield the following complexity estimation.

\begin{lemma}\label{reductionspeziell}
Let the notation be the same as in Lemma \ref{compl}. Then, Algorithm \ref{Algo2} takes $O(n^4h(T)+n^3h(T)^2)$ arithmetic operations in $k$ to transform $\bb$ into a reduced basis. 
In particular, for $h(T)\geq n$ the complexity is equal to $O(n^3h(T)^2)$.
\end{lemma}

If we use Remark \ref{denominators}, we get an estimation of
$O(n^4h(T)(\#\sig(E)+h(T))$ operations in $k$, if we do not assume that the input matrix has polynomial entries.

\section{Lattices in algebraic function fields}\label{RRChapter}
Let $F/k$ be an \emph{algebraic function field} of one variable over the constant field $k$ and let  $k_0$ be the \emph{full constant field}. That is, $F/K$ is a separable extension of finite degree $n$ and $k_0$ is the algebraic closure of $k$ in $F$.

We may realize an algebraic function field $F/k$ as the quotient field of the residue class ring $A[x]/(f(t,x))$, where 
$$f(t,x)=x^n+a_1(t)x^{n-1}+\dots+a_n(t)\in A[x]$$ 
is irreducible, monic and separable in $x$. Such a representation exists for every algebraic function field over a perfect constant field \cite[p. 128]{H.Stich}. We consider $\theta\in F$ with $f(t,\theta)=0$, so that $F=k(t,\theta)$. We call $A[\theta]$ the \emph{finite equation order} of $f$, and we define 
$$
C_f=\max\{\lceil\deg a_i(t)/i\rceil \mid 1\leq i\leq n\},\quad f_\infty(t^{-1},x)=t^{-nC_f}f(t,t^\cf x).
$$ 
Then, $f_\infty$ belongs to $k[t^{-1},x]\subset A_\infty[x]$ and the quotient field of the residue class ring $A_\infty[x]/(f_\infty(t^{-1},x))$ becomes another realization of the function field $F/k$. Clearly, $\theta_\infty:=\theta/t^\cf$ is a root of $f_\infty$. As $\theta_\infty$ is integral over $A_\infty$, we may consider the \emph{infinite equation order} $A_\infty[\theta_\infty]$.

A \emph{place} $P$ of $F/k$ is the maximal ideal of the valuation ring of 
 a surjective valuation $v_P:F\rightarrow \Z\cup\{\infty\}$, which vanishes on $k$.
Denote by $\mathbb{P}_F$ the set of all places of $F/k$ and let $\mathbb{P}_\infty(F)\subset\mathbb{P}_F$ be the set of all places over $P_\infty$. We denote $\mathbb{P}_0(F)=\mathbb{P}_F\setminus \mathbb{P}_\infty(F)$ the set of ``finite" places. 

A \emph{divisor} $D$ of $F/k$ is a formal finite $\Z$-linear combination of the places of $F$. The set $\dd_F$ of all divisors of $F/k$ is an abelian group. For a divisor $D=\sum_{P\in\pp_F}a_P P$, we set $v_P(D)=a_P$.
A partial ordering on $\dd_F$ is defined by: $D_1\leq D_2 $ if and only if $v_P(D_1)\leq v_P(D_2)$ for all $P\in \pp_F$. 

Every $z\in F^*$ determines a \emph{principal divisor} $(z)=\sum_{P\in\pp_F}v_P(z) P$. 

The \emph{Riemann-Roch space} of a divisor $D$ is the finite dimensional $k$-vector space
$$\ll(D)=\{a\in F^*\mid (a)\geq -D\}\cup \{0\}.$$ 
Instead of $\dim_k \ll(D)$, we write $\dim_k D$. 

In this section we will see that any divisor $D$ in $\dd_F$ induces a norm $\vv_D$ and a normed space $(F,\vv_D)$. Hence, the results for lattices become available in the context of algebraic function fields. 

The theory of lattices in function fields can be used to compute a $k$-basis of the Riemann-Roch space of a divisor $D$ and the successive minima of its induced lattice.
In \cite{Schoe} an algorithm is presented, which covers this problem in the context of a tamely ramified global function field. To this purpose, Puiseux expansions of certain function field elements must be computed. This leads to the technical problem of choosing the right precision of the expansions. Our algorithm for the computation of the successive minima of $D$ can be applied for arbitrary function fields and no series expansions are used.

Let $\oo_F=\mathrm{Cl}(A,F)$ and $\oo_{F,\infty}=\mathrm{Cl}(A_\infty,F)$ be the integral closures of $A$ and $A_\infty$ in $F$, respectively. These rings $\oo_F$ and $\oo_{F,\infty}$ are Dedekind domains. Hence, any nonzero fractional ideal of $\oo_F$ or $\oo_{F,\infty}$ has an unique decomposition into a product of nonzero prime ideals. 
The nonzero prime ideals of $\oo_F$ (respectively $\oo_{F,\infty}$) are in 1-1 correspondence with the finite (respectively infinite) places of $F/k$. Hence, a divisor $D$ admits a unique representation as a pair $(I,I_\infty)$ of fractional ideals $I$ of $\oo_F$ and $I_\infty$ of $\oo_{F,\infty}$. In particular, $I$ and $ I_\infty$ are $A$- and $A_\infty$-modules of full rank $n$, respectively.

More precisely, for a given divisor $D$, we consider a divisor
$$
D+r(t)_\infty=\sum_{Q\in\mathbb{P}_0(F)}\alpha_Q\cdot Q+\sum_{P\in\mathbb{P}_\infty(F)}(\beta_P+r\,e(P/P_\infty))\cdot P,$$
where $\alpha_Q,\beta_P, r\in\Z$ and $e(P/P_\infty)$ is the ramification index of $P$ over $P_\infty$. The ideal representation of $D+r(t)_\infty$ is given by $( I, t^{r}I_\infty)$, where $ I=\prod_{Q\in\mathbb{P}_0}\q^{-\alpha_Q}$ and $ I_\infty=\prod_{P\in\mathbb{P}_\infty}\p^{-\beta_P}$ constitute the ideal representation of $D$. The prime ideals $\q$ and $\p$ of $F$ are determined by the places $Q,P$ of $F$ through the identities $v_\q=v_Q$ and $v_\p=v_P$, respectively.

We consider on $F$ the norm:
\begin{align}\label{DivNorm}
\vv_D:F\rightarrow \{-\infty\}\cup\Q,\quad \v{z}_D=-\min_{P\in\mathbb{P}_\infty(F)}\left\{\frac{v_P(z)+v_P(D)}{e(P/P_\infty)}\right\}.
\end{align}
Clearly, any divisor $D$ induces a norm $\v{~}_D$. As our considerations are relative to a fixed divisor $D$, we write $\v{~}$ instead of $\v{~}_D$. 

\begin{theorem}\label{schoe}\label{dim}\quad
\begin{enumerate}
\item $\ll(D+r(t)_\infty)=I\cap t^{r}I_\infty=( I,\vv)_{\leq r}$.
\item $(I,\vv)$ is a lattice and $(F,\vv)$ is a normed space.
\end{enumerate}
\end{theorem}
\begin{proof}
We consider the first identity of item \textit{1}. For $z\in \ll(D+r(t)_\infty)$, we obtain $(z)\geq -(D+r(t)_\infty)$ and equivalently
\begin{align*}
v_Q(z)\geq -\alpha_Q, \forall Q\in\pp_0(F),\qquad v_P(z)\geq -\beta_P-r\,e(P/P_\infty), \forall P\in\pp_\infty(F).
\end{align*}
Clearly, this is equivalent to $z\in I\cap t^{r}I_\infty$. 

In order to proof the second identity of the first item we consider $z\in ( I,\vv)_{\leq r}$. That is, $z\in I$ with $\v{z}\leq r$, which is equivalent to
$$
\min_{P\in\pp_\infty(F)}\Big\{\frac{v_P(z)+v_P(D)}{e(P/P_\infty)}\Big\}\geq -r\ \Longleftrightarrow v_P(z)+\beta_P\geq -re(P/P_\infty),\forall P\in\pp_\infty(F).
$$
This is equivalent to $z\in I\cap t^{r}I_\infty$, since $z\in I$.

We consider the second item. Regarding Definition \ref{deflat}, we have to show that $\dim_k (I,\vv)_{\leq r}<\infty$, for all $r\in \R$. This follows directly from item \textit{1}. 
\end{proof}

By the last theorem we can identify any divisor $D$ uniquely with the lattice $(I,\v{~})$. Hence, we can define the successive minima $\sm(D)$ of $D$ to be the successive minima of the corresponding lattice. We call two divisors $D_1$ and $D_2$ \emph{isometric} if they have the same successive minima, and we write then $D_1\sim D_2$. Clearly $\sim$ is an equivalence relation on the set of divisors. The class of $D$ in $\dd_F/\!\sim$ is called the \emph{isometry class} of $D$.

\begin{corollary}\label{basisRRSpace}
Let $\bb=(b_1,\dots,b_n)$ be a semi-reduced basis of $( I,\vv)$. Then,
\begin{enumerate}
\item $ I_\infty=(F,\vv)_{\leq 0}$,
\item the set $\{b_it^{j_i}\mid 1\leq i\leq n,\  0\leq j_i\leq- \lceil \v{b_i}\rceil +r\}$ is a k-basis of $\ll(D+r(t)_\infty)$,
\item $\dim_k (D+r(t)_\infty)=\sum_{\lceil \v{b_i}\rceil\leq r}(-\lceil \v{b_i}\rceil+r +1)$.
\end{enumerate}
\end{corollary}

\begin{proof}
By Theorem \ref{schoe}, $(I,\vv)$ is a lattice. Let  $m_i=-\lceil\v{b_i}\rceil$, for $1\leq i\leq n$. By Lemma \ref{equivsemiot}, the family $(t^{m_1}b_1,\dots,t^{m_n}b_n)$ is a semi-orthonormal basis of $(F,\vv)$. Hence, Lemma-Definition \ref{smiredundred} yields the first item of the theorem.

Since $\bb$ is a reduced basis of $( I,\lceil \vv\rceil)\sim ( I,\vv)$, the second item follows from Theorem \ref{semisuccmin} and Proposition \ref{basic}. The third one follows from the second one.
\end{proof}

The successive minima $\sm(D)$ determine the isometry class of a divisor $D$. For the computation of $\sm(D)$ we need a reduced basis $\bb$ of the corresponding lattice $(I,\v{~})$. According to Subsection \ref{semisection} an orthonormal basis $\bb'$ of the normed space $F$ is required. If $\mathrm{supp}(D)\cap \pp_\infty(F)=\emptyset$, algorithms  which determine a reduced basis of $(F,\v{~})$ can be found in \cite{integralPaper}, \cite{NartBase} and \cite{Stain}. These ideas can be easily generalized to arbitrary divisors $D$, for instance see \cite{phd}. We assume that a basis $\bb$ of $I$ and an orthonormal basis $\bb'$ of $(F,\v{~})$ are already available. Then, Algorithm \ref{Algo2} transforms $\bb$ into a reduced basis of $(I,\v{~})$. Since every reduced basis is in particular semi-reduced by Corollary \ref{basisRRSpace}, Algorithm \ref{Algo2} determines a basis of the Riemann-Roch space $\ll(D)= I \cap  I_\infty$ too.

In \cite[Theorem 3.2]{integralPaper} it is shown that the semi-reduced bases of $(F,\v{~})$ are characterized by the $A_\infty$-bases of the fractional ideal $I_\infty$. For a basis $\bb'$ of $I_\infty=(F,\v{~})_{\leq 0}$, which is not reduced, Algorithm \ref{Algo2} does compute a $k$-basis of $\ll(D)$ but not the successive minima of $D$. If we call in that context the simplified reduction Algorithm \ref{Algo0} for the transition matrix $T(\bb\rightarrow\bb')$ then we are in the case of Hess' algorithm described in \cite{F.H.}. Hence, Algorithm \ref{Algo2} can be considered as a refinement of Hess' algorithm in that setting.

\section{Appendix: Complexity of the computation of the successive minima}\label{complexirrspaci}

\subsection{Bases of fractional ideals}\label{ideals}
 Let $R$ be either $A$ or $A_\infty$. We denote $\oo_R=\oo_F$, $\t_R=\t$, if $R=A$, and $\oo_R=\oo_{F,\infty}$, $\t_R=\t_\infty$, if $R=A_\infty$.
We consider $\bb_{\t_R}=(1,\t_R,\dots, \t_R^{n-1})$, which is a basis of $R[\t_R]$. 

  Let $M$ and $M'$ be two free $R$-modules of rank $n$. The \emph{index} $[M:M']$ is the nonzero fractional ideal of $R$ generated by the determinant of the transition matrix from a basis of $M'$ to a basis of $M$.

In the sequel we consider canonical bases of fractional ideals in function fields. These canonical bases consist of elements having ``small" size, which is comfortable from the computational point of view. Moreover, we determine concrete bounds for the entries of the transition matrix from such a canonical basis to $\bb_{\t_R}$. 

Let $I=\prod_{\p\in\ma(\oo_R)}\p^{a_\p}$ be a nonzero fractional ideal of $\oo_R$. We define
\begin{align}\label{istern}
 I^*:=\prod_{\p\in\ma(\oo_R)}\p^{-|a_\p|}.
\end{align}
Clearly, $I^*$ is again a fractional ideal of $\oo_R$.

For $h\in K$ we set $|hR|:=|h|$ and extend the degree function $|~|$ to fractional ideals of $R$.
\begin{definition}\label{deffiofheight}
The \emph{height} of the fractional ideal $I$ of $\o1$ or $I_\infty$ of $\oi$ is defined to be the integer
$$
h(I)=|[ I^*:A[\theta]]|\quad \text{or}\quad h( I_\infty)=-|[ I^*_\infty:A_\infty[\theta_\infty]]|.
$$
Additionally, we define the \emph{absolute height} of $I$ or $I_\infty$ by
$$
H(I)=|[ I^*:\o1]|+|\dsc f|\quad \text{or}\quad H( I_\infty)=-|[ I^*_\infty:\oi]|-|\dsc f_\infty|.
$$
\end{definition}
\begin{lemma}\label{heightlamma}
Let $I$ and $I_\infty$ be as in the last definition. Then, it holds
\begin{enumerate}
\item $h(I),\ h(I_\infty),\ H(I),\ H(I_\infty)\geq 0$,
\item $h(I)\leq |[I^*:\o1]|+\frac{1}{2}|\dsc f|\leq H(I)$, 
\item $h(I_\infty)\leq -|[I^*_\infty:\oi]|-\frac{1}{2}|\dsc f_\infty|\leq H(I_\infty)$.
\end{enumerate}
\end{lemma}
\begin{proof}
Since the exponents in the decomposition of $I^*$ and $I_\infty^*$ are nonpositive integers, we have $A[\theta]\subseteq\o1\subseteq I^*$ and $A_\infty[\theta_\infty]\subseteq\oi\subseteq I_\infty^*$. Then, by the properties of the index of modules we deduce $[ I^*:A[\theta]]=r A$ with $r\in A$ and $[ I_\infty^*:A_\infty[\theta_\infty]]=r' A_\infty$ with $r'\in A_\infty$; hence, $h(I)=|r|\geq 0$ and $h(I_\infty)=-|r'|\geq 0$. Since $|\dsc f|,\ -|\dsc f_\infty|\geq 0$, we deduce $ H(I),\ H(I_\infty)\geq 0$.

For the second statement we use the transitivity of the index 
$$[ I^*:A[\theta]]=[ I^*:\o1][ \o1:A[\theta]].
$$ 
Denote by $\bb=(b_0,\dots,b_{n-1})$ a basis of $\oo_F$ and let $\bb_\t=(1,\t,\dots,\t^{n-1})$. By \cite{Zass} it holds
$$
\dsc f=\det (\Tr(\t^{i+j}))_{0\leq i,j<n}=(\det T(\bb_\t\rightarrow \bb))^2\cdot \det (\Tr(b_ib_j))_{0\leq i,j<n}.
$$
Then, $(\det T(\bb_\t\rightarrow \bb))^2$ divides $\dsc f$ and therefore $(\dsc f )A\subset (\det T(\bb_\t\rightarrow \bb))^2A=[\o1:A[\t]]^2$. Hence, $|(\dsc f )A|=|\dsc f|\geq 2|[\o1:A[\t]]|$, and in particular $|[ I^*:A[\theta]]|=|[ I^*:\o1][ \o1:A[\theta]]|\leq |[ I^*:\o1]|+\frac{1}{2}|\dsc f|\leq H(I)$. Item \textit{3} can be shown analogously.
\end{proof}

\begin{definition}
Let $\bb$ be a basis of a fractional ideal $I$ of $\oo_R$ and $T$ the transition matrix from $\bb$ to $\bb_{\t_R}$. We call $\bb$ an \emph{Hermite basis} of $I$, if the matrix $hT$ is in Hermite normal form (HNF), for any $h\in R\setminus R^*$ such that $hT\in R^{n\times n}$.
\end{definition}

\begin{lemma}
Every ideal $I$ of $\oo_R$ admits a unique Hermite basis.
\end{lemma}

Let $\bb$ be an Hermite basis of $I$ and $T=T(\bb\rightarrow \bb_{\t_R})$. The diagonal entries $d_1,\dots,d_n\in K$ of $T$ are canonical invariants of the fractional ideal $I$, which only depend on $f$, the defining polynomial of $F/k$. In particular, $[R[\t_R]:I]=(d_1\cdots d_n )R$. 

From the fact that $I$ is an ideal we deduce $d_n|\cdots |d_1$; that is, $d_i/d_{i+1}\in R$ for all $i$. We call these elements the \emph{elementary divisors} of $I$. If $I$ is contained in $R[\t_R]$, we obtain $d_1,\dots, d_n\in R$ and
$$
R[\t_R]/I\cong R/d_1R\times \cdots \times R/d_nR.
$$
For any subset $S\subset R$, we call an element $h\in S\setminus\{0\}$ minimal if $\deg h$ or $v_\infty(h)$ is minimal among all other elements in $S$, for $R=A$ or $R=A_\infty$, respectively.

\begin{lemma}\label{hightideal}
Let $\bb$ be an Hermite basis of a fractional ideal $I$ of $\oo_R$ and $(t_{i,j})=T(\bb\rightarrow \bb_{\t_R})$. For $g\in R$ minimal such that $gT\in R^{n\times n}$ it holds, 
$$
|gt_{i,j}|\leq  H(I)\quad\text{or}\quad v_\infty(gt_{i,j})\leq H(I)
$$
according to $R=A$ or $R=A_\infty$.
\end{lemma}
In order to proof this statement we will use the following lemma.

\begin{lemma}\label{helpbigh}
For $I=\prod_{\p\in\ma(\oo_R)}\p^{a_\p}$, we write $I=I_1\cdot I_2$, where $I_1=\prod_{a_\p<0}\p^{a_\p}$ and $I_2=\prod_{a_\p>0}\p^{a_\p}$. Then, 
$$
h(I_1)+h(I_2)\leq H(I).
$$
\end{lemma}
\begin{proof}
Let $R=A$, the case $R=A_\infty$ can be treated analogously. Clearly, $I^*=I_1^*\cdot I_2^*$ by definition. Then, 
\begin{align*}
[I^*:\o1][\o1:A[\t]]^2=[I_1^*:A[\t]][I_2^*:A[\t]].
\end{align*}
According to the proof of Lemma \ref{heightlamma} we have $|[\o1:A[\t]]|\leq \frac{1}{2}|\dsc f|$; hence, $h(I_1)+h(I_2)= |[I^*:\o1]|+2|[\o1:A[\t]]|\leq H(I)$.
\end{proof}

\begin{proof}[Proof of Lemma \ref{hightideal}]
We consider the case $R=A$. The case $R=A_\infty$ can be treated analogously. Let $I=I_1\cdot I_2$ with $I_1,I_2$ defined as in Lemma \ref{helpbigh}.

As $[I_1 : A[\theta]]=rA$ with $r\in A$, we deduce $|g| \leq |r| = |[I_1 : A[\theta]]|=h(I_1)$, by the minimality of $|g|$.

Since $\bb$ is an Hermite basis of $I$, the matrix $T:=T(\bb\rightarrow \bb_{\t_R})$ is triangular and the entries of the $j$-th column satisfy $|t_{i,j}|\leq |t_{j,j}|$, for $ j\leq i\leq n$. We consider the matrix $T^{-1}=T(\bb_{\t_R}\rightarrow \bb)$ which has the diagonal entries $t_{j,j}^{-1}$ for $1\leq j\leq n$. Let $g'\in A\setminus\{0\}$ be of minimal degree such that $g'A[\t]\subset I$ (or equivalently $g'T^{-1}\in A^{n\times n}$). Then, $|g't_{j,j}^{-1}|\geq 0$ and equivalently $|g'|\geq|t_{j,j}|$. 

Since  $[\o1:I_2]=r'A$ with $r'\in A$, we obtain $|g'|\leq|r'|= |[\o1:I_2]|$ by the minimality of $|g'|$. Now, $[\o1:I_2]=[\o1 : (I_2^*)^{-1}]=[I_2^* : \o1]$, so that $|g'|\leq |[I_2^* : \o1]|\leq h(I_2)$.

Finally, we deduce $|gt_{i,j}| \le h(I_1)+h(I_2) \le H(I)$, by Lemma \ref{helpbigh}. 
\end{proof}

\begin{corollary}\label{hightideal2}
Let $\bb$ be an Hermite basis of a fractional ideal $I$ of $\oo_R$. Suppose that $T=T(\bb\rightarrow \bb_{\t_R})=(f_{i,j}/h_{i,j})$ with coprime polynomials $f_{i,j},h_{i,j}\in A$, and let $g\in A\setminus\{0\}$ be of minimal degree such that $gT\in A^{n\times n}$. For $1\leq i,j\leq n$, we have
$$
|g|+\max\{|f_{i,j}|,|h_{i,j}|\}\leq 2H(I).
$$

\end{corollary}
\begin{proof}
For $R=A$ the statement is a direct consequence of Lemma \ref{hightideal}.

Let $R=A_\infty$. We consider the elementary divisors $d_1,\dots,d_n$ of $I$, which satisfy $d_i=t^{\alpha_i}$ with $\alpha_i\in \Z$ and $\alpha_1\leq \cdots\leq \alpha_n$, since $d_{i}/d_{i+1}\in A_\infty$ for all $i$. We fix $g'=t^{-\beta}$, where $\beta=\max\{\alpha_n,0\}$. Then, $g'\in A_\infty$ is minimal with $g'T\in A_\infty^{n\times n}$. Lemma \ref{hightideal} shows that $v_\infty(g't_{i,j})\leq  H(I)$, where $
(t_{i,j})=T$. Since $g'T$ is in HNF, the diagonal entries are $t^{-1}$-powers, and in particular $v_\infty(g't_{i,i})=\deg_{t^{-1}}(g't_{i,i})$ holds. Hence, the entries in $g'T$ satisfy $\deg_{t^{-1}}(g't_{i,j})\leq H(I)$ by the definition of the HNF in that context. For any $h\in k[t^{-1}]$ of $t^{-1}$-degree equal $m$ we can write $h=t^mh/t^m$ with $t^mh\in A$ and $|t^mh|\leq m$. Thus, any entry of $g'T$ can be written as $f_{i,j}/t^{m_{i,j}}$ with $f_{i,j}\in A$, $|f_{i,j}|\leq H(I)$, and $0\leq m_{i,j}\leq H(I)$. Clearly, there exists $m\in \Z$, with $m\leq H(I)$, such that $t^mf_{i,j}/t^{m_{i,j}}\in A$, for all $1\leq i,j\leq n$. We set $g:=t^m$ and $h_{i,j}:=t^{m_{i,j}}$ and obtain $|g|+\max\{|f_{i,j}|,|h_{i,j}|\}\leq 2H(I)$, for $1\leq i,j\leq n$.
\end{proof}

\subsection{complexity}
In \cite{Diem,Eisen} it is shown that the computation of a basis of a Riemann-Roch space $\ll(D)$ by Hess' algorithm is polynomially bounded in $n$ and the \emph{height} $h(D)$ of $D$ (see definition below). In the sequel we are going to give precise bounds for the complexity of the computation of the successive minima of a divisor $D$ and therefore for the computation of a basis of $\ll(D)$ by Algorithm \ref{Algo2}. We fix a divisor $D$ and denote by $(I,I_\infty)$ its ideal representation; that is, $\ll(D)=I\cap I_\infty$. We need a basis $\bb$ of the fractional ideal $I$ and a reduced basis $\bb'$ of the normed space $(F,\v{~})$. We assume that $\bb$ is a Hermite basis and that $\bb'$ is obtained by the algorithm explained in \cite{Stain}; that is, $\bb'$ is given by a triangular basis of the fractional ideal $I_\infty$, which is a reduced basis of $(F,\v{~})$. Note that in \cite{Stain} the basis $\bb'$ is constructed such that the entries of the transition matrix $T(\bb\rightarrow \bb_{\t_\infty})$ 
satisfy the bounds from Lemma \ref{hightideal}. Thus, for the complexity estimation we can assume that $\bb'$ is a Hermite basis of $I_\infty$. Note that in general a Hermite basis of $I_\infty$ is not reduced.

Denote by $T$ the transition matrix from $\bb$ to $\bb'$. Then, the rows of $T$ are given by the coordinate vectors $c_{\bb'}(b)$ for $b\in \bb$. 

Let $g_1,\dots,g_n \in A$ be nonzero polynomials of minimal degree such that $\widetilde{T}:=T\cdot \dia(g_1,\dots,g_n)\in A^{n\times n}$. Denote by $C(T)$ the cost of the computation of the transition matrix $T$. Then, the following statement follows immediately from Lemma \ref{reductionspeziell}.

\begin{lemma}\label{costrrcom}
Algorithm \ref{Algo2} needs at most
$$
C(T)+O(n^3h(\widetilde{T})(n+h(\widetilde{T})))
$$
arithmetic operations in $k$ to determine the successive minima of $D$ and a basis of $\ll(D)$.
\end{lemma}

We are interested in a complexity estimation, which only depends on the data $n$ and $\cf$ of the defining polynomial $f$ of the function field and on the divisor $D$ (cf. Corollary \ref{complexRR}). Therefore, we estimate $h(\widetilde{T})$ and $C(T)$ in terms of $n,\cf$ and $h(D)$ (see below).

\begin{definition}[Divisor height]
For $D\in \dd_F$, we define \emph{the height} of $D$ by $h(D)=\deg  D^*$, where
$$
D^*= \sum_{P\in\pp_F}|v_P(D)|\cdot P.
$$
\end{definition}
Note that the height of a divisor is a nonnegative integer and $h(D)=0$ if and only if $D=0$.

We will now formulate some technical lemmas, which will be useful for further complexity estimations. 

\begin{lemma}\label{boundy}
Let $F/k$ be a function field with defining polynomial $f$ of degree $n$. Then, $\delta:=|\dsc f |$ and $\delta_\infty:=v_\infty(\dsc f_\infty )$ satisfy
$$
\delta,\delta_\infty\leq \delta+\delta_\infty=\cf n(n-1)=O(n^2\cf).
$$
In particular, it holds $|[\oo_F:A[\theta]]|\leq \delta$ and $-|[\oo_{F,\infty}:A_\infty[\theta_\infty]]|\leq \delta_\infty$.
\end{lemma}
\begin{proof}
See \cite[Lemma 3.8]{GenCom}.
\end{proof}

\begin{lemma}\label{hightbound}\quad
\begin{enumerate}
\item $h( I)+h( I_\infty)\leq H( I)+H( I_\infty) =O(h(D)+n^2\cf).$
\item $h(\widetilde{T})=O(nh(D)+n^3\cf)$.
\end{enumerate}
\end{lemma}
\begin{proof}
In order to prove the first item we consider $D=\sum_{P\in\pp_F}a_PP$ and set $D_0=\sum_{P\in\pp_0(F)}a_PP$ and $D_\infty=\sum_{P\in\pp_\infty(F)}a_PP$. The ideal representation of $D$ is given by $(I,I_\infty)$ with
$$
I=\prod_{P\in\pp_0(F)}\p^{-a_P},\quad I_\infty=\prod_{P\in\pp_\infty(F)}\p^{-a_P},
$$
where the prime ideals $\p$ of $F$ corresponds to the places $P$ of $F$. We consider $D^*_0=\sum_{P\in\pp_0(F)}|a_P|P$ and $D^*_\infty=\sum_{P\in\pp_\infty(F)}|a_P|P$ and set $I^*=\prod_{P\in\pp_0(F)}\p^{-|a_P|}$ and $I^*_\infty:=\prod_{P\in\pp_\infty(F)}\p^{-|a_P|}$ as in (\ref{istern}). It is well known  that
$$
[\o1:I^*]=N_{F/K}(I^*)=\prod_{P\in\pp_0(F)}N_{F/K}(\p)^{-|a_P|}.
$$
Since $\deg  P=|N_{F/K}(\p)|$ \cite{H.H.}, we obtain, $|[\o1:I^*]|=\sum_{P\in\pp_0(F)} -|a_P|\deg  P=-\deg  D^*_0$. As $|[\o1:I^*]|=-|[I^*:\o1]|$, we get
\begin{align}\label{degindexform}
\deg  D^*_0=|[I^*:\oo_F]|=|[I^*:A[\t]]|-|[\o1:A[\t]]|.
\end{align}
Analogously, one can show $\deg  D^*_\infty=-|[I^*_\infty:A_\infty[\t_\infty]]|+|[\oi:A_\infty[\t_\infty]]|$. Then, by the definition of the height of an ideal (cf. Definition \ref{deffiofheight}) and of a divisor, we obtain $\deg  D^*_0=h(D_0)=h( I)-|[\oo_F:A[\theta]]|$ and $\deg  D^*_\infty= h(D_\infty)=h( I_\infty)+|[\oo_{F,\infty}:A_\infty[\theta_\infty]]|$. Since the supports of $D_0$ and $D_\infty$ are disjoint, we obtain
$$
h(D)=h(D_0)+h(D_\infty)=h( I)-|[\oo_F:A[\theta]]|+h( I_\infty)+|[\oo_{F,\infty}:A_\infty[\theta_\infty]]|
$$
and therefore $h( I)+h( I_\infty)\leq h(D)+\delta+\delta_\infty$. Clearly, $H(I)\leq h(I)+\delta$ and $H(I_\infty)\leq h(I_\infty)+\delta_\infty$ (cf. Definition \ref{deffiofheight}). Thus, we deduce $H( I)+H( I_\infty)\leq h(I)+\delta+h(I_\infty)+\delta_\infty\leq h(D)+2(\delta+\delta_\infty)=O(h(D)+n^2\cf)$ by Lemma \ref{boundy}.

We consider the second item: For a matrix $N\in K^{n\times n}$ denote by $g_N\in A$ a nonzero polynomial of minimal degree such that $g_NN\in A^{n\times n}$.
Then, by the definition of $\widetilde{T}$ we have $h(\widetilde{T})\leq h(g_TT)$. Let us estimate the height of $g_TT$.

We consider the matrices $M,M'\in K^{n\times n}$ with $M(1\ \theta \dots \theta ^{n-1})^\tp=(b_1 \dots b_n)^\tp$ and $M'(1\ \theta \dots \theta ^{n-1})^\tp=(b'_1\dots b'_n)^\tp$, where $\bb=(b_1,\dots,b_n)$ and $\bb'=(b'_1,\dots,b'_n)$. Then, $T=MM'^{-1}$ is the transition matrix from $\bb$ to $\bb'$. Clearly, $|g_T|\leq |g_M|+|g_{M'^{-1}}|$, since $g_Mg_{M'^{-1}}T\in A^{n\times n}$ and $|g_T|$ is minimal. Then, Lemma \ref{hightinversemat} shows that
\begin{align}\label{gg1g2}
h(\widetilde{T})\leq h(g_TT)= |g_T|+h(T)\leq  |g_M|+h(M)+|g_{M'^{-1}}|+h(M'^{-1}).
\end{align}
As $\bb$ is an Hermite basis, Corollary \ref{hightideal2} shows that $|g_M|+h(M)=O(H( I))$. We estimate $|g_{M'^{-1}}|+h(M'^{-1})$ and consider
$$
M'\dia(1, t^{\cf}, \dots, t^{(n-1)\cf })(1\ \theta_\infty\dots\theta_\infty^{n-1})^\tp=(b'_1\dots b'_n)^\tp.
$$
We set $Q:=M'\dia(1, t^{\cf}, \dots, t^{(n-1)\cf })$. As $M'^{-1}=\dia(1, t^{\cf}, \dots, t^{(n-1)\cf })Q^{-1}$, we obtain 
\begin{align}\label{keinename}
|g_{M'^{-1}}|+h(M'^{-1})\leq |g_{M'^{-1}}|+(n-1)\cf +h(Q^{-1}).
\end{align}
As $g_{Q^{-1}}M'^{-1}\in A^{n\times n}$, we deduce $|g_{M'^{-1}}|\leq |g_{Q^{-1}}|$. Arguing as we did in the proof of item \textit{3} of Lemma \ref{hightinversemat}, we see that $(g_Q^n\det Q)Q^{-1}\in A^{n\times n}$ with $g_Q^n\det Q\in A$; hence $|g_{Q^{-1}}|\leq|g_Q^n\det Q| \leq n(g_Q+h(Q))$. Moreover, we have $h(Q^{-1})\leq nh(Q)$. As $Q$ is the transition matrix from $\bb'$ to $(1,\t_\infty,\dots,\t_\infty^{n-1})$, it holds $|g_Q|+h(Q)=O(H(I_\infty))$ by Corollary \ref{hightideal2} and therefore $|g_{M'^{-1}}|+h(M'^{-1})=O(nH(I_\infty)+n\cf)$ by (\ref{keinename}). Finally, (\ref{gg1g2}) and item \textit{1} show that
\begin{align}\label{smallhight}
h(\widetilde{T})=O(H( I)+n\cf+nH(I_\infty))=O(nh(D)+n^3\cf).
\end{align} 

\end{proof}

\begin{corollary}\label{complexRR}
Let $D$ be a divisor with $\mathcal{L}(D)= I \cap  I_\infty$ and $\bb$ and $\bb'$ as above. Then, Algorithm \ref{Algo2} needs at most 
$$O(n^5(h(D)+n^2\cf)^2)$$ 
arithmetic operations in $k$ to compute $\sm(D)$.
\end{corollary}


\begin{proof}
We apply Lemma \ref{hightbound} to Lemma \ref{costrrcom} and deduce that the complexity of Algorithm \ref{Algo2} is given by
$$
C(T)+O(n^3h(\widetilde{T})(n+h(\widetilde{T})))=C(T)+O(n^5(h(D)+n^2\cf)^2).
$$
In order to estimate $C(T)$ we consider the proof of Lemma \ref{hightbound}. There we have seen that $T=MM'^{-1}$. Clearly, the cost $C(T)$ for computing $T$ is dominated by the cost of the inversion of $M'$ and the realization of the matrix product $MM'^{-1}$.

Since $M'=Q\dia(1, t^{-\cf}, \dots, t^{-(n-1)\cf })$, the cost for determining $M'^{-1}$ is dominated by the inversion of $Q$. As mentioned above we can assume that $\bb'$ is a Hermite basis of $\ii_\infty$; that is, there exist $\beta \in \Z$ such that $t^\beta Q$ is in HNF. We can assume that $\beta=0$. Hence, we have to invert a lower triangular matrix, whose entries $q_{i,j}$ satisfy $|q_{i,j}|=O(h(I_\infty))$ by Corollary \ref{hightideal2}. By Gaussian elimination this can be realized with at most $O(n^3h(I_\infty))$ operations in $k$.

Since $h(g_MM)=O(h(I))$ and $h(g_{M'^{-1}}M'^{-1})=O(n\cf+nh(I_\infty))$, the cost for computing $MM'^{-1}$ is bounded by 
$O(n^3(n\cf+nh(I_\infty)+h(I))^2)$ operations in $k$. Hence, $C(T)$ is dominated by $O(n^5(h(D)+n^2\cf)^2)$.
\end{proof}

In the sequel we assume that the constant field $k$ is finite with $q$ elements and we admit fast multiplication techniques of Sch\"onhage-Strassen \cite{Gath}. Let $R$ be a ring and let $g_1,g_2\in R[x]$ be two polynomials, whose degrees are bounded by $d_1$ and $d_2$, respectively. Then, the multiplication $g_1\cdot g_2$ needs at most $O(\max\{d_1,d_2\}^{1+\epsilon})$ operations in $R$. 

 \begin{theorem}\label{complexrrspaci}
 Let $F/k$ be a function field with defining polynomial $f$ of degree $n$ and let $D=\sum_{P\in\pp_F}a_PP$ be a divisor of $F/k$. Then, the successive minima of $D$ and a $k$-basis of $\ll(D)$ can be determined with
 $$
O(n^5(h(D)+n^2\cf)^2+n^{5+\epsilon}\cf^{2+\epsilon}\log q)
 $$
 operations in $k$.
 \end{theorem}
\begin{proof}
Let $(I,I_\infty)$ be the ideal representation of $D$. In order to determine a $k$-basis of $\ll(D)$ we compute a Hermite basis $\bb$ of $I$, a reduced basis $\bb'$ of $(F,\v{~})$, and apply Algorithm \ref{Algo2}.

By Lemma \ref{boundy} we have $\delta+\delta_\infty =O(n^2\cf)$. Moreover, Lemma \ref{hightbound} shows that $H(I)+H(I_\infty)=O(h(D)+n^2\cf)$. By \cite[Theorem 5.3.19, Corollary 5.3.14]{phd} the computation of $\bb$ and $\bb'$ takes $O(n^3H(I)^2+n^{1+\epsilon}\delta^{2+\epsilon}\log q)$ and $O(n^{2+\epsilon}H(I_\infty)^{1+\epsilon}+n^{1+\epsilon}\delta_\infty \log(q)+n^{1+\epsilon}\delta_\infty^{2+\epsilon})$ operations in $k$, respectively. Together we deduce

\begin{align*}
&\ O(n^3(H(I))^2+n^{1+\epsilon}\delta^{2+\epsilon}\log q+    n^{2+\epsilon}H(I_\infty)^{1+\epsilon}+n^{1+\epsilon}\delta_\infty \log(q)+n^{1+\epsilon}\delta_\infty^{2+\epsilon})\\
=&\ O(n^3(H(I)+H(I_\infty))^2+n^{1+\epsilon}(\delta+\delta_\infty)^{2+\epsilon}\log q)\\
=&\ O(n^3(h(D)+n^2\cf)^2+n^{5+\epsilon}\cf^{2+\epsilon}\log q)
\end{align*}
operations in $k$.

Additionally, we run Algorithm \ref{Algo2}, which needs $O(n^5(h(D)+n^2\cf)^2)$ operations in $k$ by Corollary \ref{complexRR}. Together we can estimate the computation of $\sm(D)$ and a $k$-basis of $\ll(D)$ by
$$
O((n^5(h(D)+n^2\cf)^2+n^{5+\epsilon}\cf^{2+\epsilon}\log q))
$$
operations in $k$.
\end{proof}

\begin{corollary}\label{complexRRzero}
For a divisor $D$, let $D=D_0+D_\infty$ as defined in the proof of Lemma \ref{hightbound}. If there exists an integer $r$ such that $D_\infty=r(t)_\infty$, then the successive minima of $D$ and a $k$-basis of $\ll(D)$ can be determined with
 $$
O(n^3(h(D)+n^3\cf)^2+n^{5+\epsilon}\cf^{2+\epsilon}\log q)
 $$
 operations in $k$.
\end{corollary}
\begin{proof}
Denote by $(s_1,\dots,s_n)$ and $(s'_1,\dots,s'_n)$ the successive minima of $D$ and $D_0$, respectively. Clearly, $D_\infty=r(t)_\infty$ implies $s_i=s'_i+r$ for $i=1,\dots,n$. Moreover by Corollary \ref{basisRRSpace} it is sufficient to determine a reduced basis of the lattice induced by $D_0$ in order to deduce a basis of $\ll(D)$. Hence, we can assume that $r=0$. Then, the ideal representation of $D$ is given by $(I,I_\infty)$ with $I_\infty = \oo_{F,\infty}$. Let $\widetilde{T}$ be defined as above. We consider (\ref{smallhight}) with $I_\infty = \oo_{F,\infty}$. Then, $h(\widetilde{T})=O(H(I)+n\cf+n\delta_\infty)=O(H(I)+n^3\cf)$ by the definition of $H(I_\infty)$ and Lemma \ref{boundy}. We apply Lemma \ref{hightbound} and deduce $h(\widetilde{T})=O(h(D)+n^3\cf)$. If we replace the bound for $h(\widetilde{T})$ in the proof of Corollary \ref{complexrrspaci} by the new one, we deduce the complexity bound from the statement.

\end{proof}


\begin{thebibliography}{}


\bibitem{phd}
J.-D. Bauch, \emph{Lattices over polynomial Rings and Applications to Function Fields}, PhD thesis, Universitat Aut`onoma de Barcelona, July 2014.


\bibitem{GenCom}
J.-D. Bauch, \emph{Genus computation of global function fields}, J. Symbolic Computation, {\bf 66} (2015), 8-20.

\bibitem{integralPaper}
J.-D. Bauch, \emph{Computation of integral basis},  	arXiv:1507.04058v3 [math.NT], November 2015.

\bibitem{H.C.}
H. Cohen, \emph{A Course in Computational Algebraic Number Theory},  Graduate Texts in Mathematics, Springer, 1993.

\bibitem{HCNH}
H. Cohn and N. Heninger, \emph{Approximate common divisors via lattices}, Proceedings ANTS X, 271- 293, 2012.


\bibitem{Diem}
C. Diem. \emph{On arithmetic and the discrete logarithm problem in class groups of curves}, Habilitationsschrift, Universit\"at Leipzig. Available at \url{http://www.math.uni- leipzig.de/?diem/preprints/habil.pdf}, 2008.

\bibitem{Eisen}
K. EisentrŠger, S. Hallgren, \emph{Computing the unit group, the class group and compact representations in algebraic function fields}, ANTS X. Proceedings of the tenth algorithmic number theory symposium, 335-358, 2013.


\bibitem{NartBase}
J. Guardia, J. Montes, E. Nart, \emph{Higher Newton polygons and integral bases}, J. Number Theory 147 (2015), 549Ð589.

\bibitem{H.H.}
H. Hasse, \emph{Number theory}, second Edition, Grundlehren der mathematischen Wissenschaften, Springer, 1980.

\bibitem{F.H.}
F. Hess, \emph{Computing Riemann-Roch spaces in algebraic function fields and related topics}, J. Symbolic Computation, 11,1-000, 2011.


\bibitem{Len}
A. Lenstra, \emph{Factoring Multivariate Polynomials over Finite Fields}, J. Computer and System Science {\bf 30} (1985), 235-248.


\bibitem{Code}
R. J. McEliece, \emph{The algebraic theory of convolutional codes} In V. Pless and W. Huffman, editors, Handbook of Coding Theory, Vol. 1, p. 1065Ð1138. Elsevier, Amsterdam, 1998.


\bibitem{MA}
K. Mahler, \emph{An analogue to Minkowski's geometry of numbers in a field of series}, Ann. Math. (2) {\bf 42} (1941), 488-522.
  
\bibitem{MS}
T. Mulders and A. Storjohann, \emph{On lattice reduction for polynomial matrices}, J. Symbolic Computation {\bf 35} (2003), 377-401.




\bibitem{Zass}
M. Pohst, H. Zassenhaus, \emph{Algorithmic Algebraic Number Theory}, Encyclopedia of Mathematics and its Applications, Cambridge University Press, 1997.



\bibitem{Sch}
W. M. Schmidt, \emph{Construction and estimation of bases in function fields}, J. Number Theory {\bf 39} (1991), 181-224. 


\bibitem{Schoe}
M. Sch\"ornig, \emph{Untersuchung konstruktiver Probleme in globalen Funktionenk\"orpern}, Technische Universit\"at Berlin, 1996.


\bibitem{Stain}
H. D. Stainsby, \emph{Triangular bases of integral closures},  arXiv:1506.01904v2 [math.NT], June 2015.
  

\bibitem{H.Stich}
H. Stichtenoth, \emph{Algebraic Function Fields and Codes}, second Edition, Graduated Texts in Mathmatics, Springer, 2008.

 
\bibitem{Gath}
J. von zur Gathen and J. Gerhard, \emph{Modern Computer Algebra}, Cambridge University Press (2003), ISBN={9780521826464}.




\end{thebibliography}
\end{document}